\documentclass[11pt,reqno]{amsart}
\usepackage{amsmath,amsthm,amssymb,bm,amscd}
\usepackage[hypertex]{hyperref}

\makeatletter
    
    \@addtoreset{equation}{section}
\makeatother

\newtheorem{Thm}{Theorem}[section]
\newtheorem{Lem}[Thm]{Lemma}

\newtheorem{Prop}[Thm]{Proposition}
\newtheorem{Rem}[Thm]{Remark}
\newtheorem{Def}[Thm]{Definition}

\newcommand{\C}{\mathbb{C}}           % Use for complex numbers.
\newcommand{\Z}{\mathbb{Z}}

\newcommand{\ad}{\text{ad}}

\newcommand{\id}{\mathrm{id}}

\newcommand{\fa}{{\mathfrak a}}             % more Fraktur
\newcommand{\fb}{{\mathfrak b}}
\newcommand{\fg}{{\mathfrak g}}
\newcommand{\fh}{{\mathfrak h}}

\newcommand{\fn}{{\mathfrak n}}
\newcommand{\fp}{{\mathfrak p}}

\newcommand{\hfh}{\hat{\fh}}
\newcommand{\hfb}{\hat{\fb}}

\newcommand{\hfp}{\hat{\fp}}

\newcommand{\hI}{\hat{I}}
\newcommand{\hW}{\hat{W}}

\newcommand{\ga}{\alpha}
\newcommand{\gb}{\beta}

\newcommand{\gl}{\lambda}
\newcommand{\gL}{\Lambda}
\newcommand{\gd}{\delta}
\newcommand{\gD}{\Delta}

\newcommand{\gt}{\theta}

\newcommand{\gG}{\Gamma}
\renewcommand{\ggg}{\gamma}
\newcommand{\gs}{\sigma}

\renewcommand{\hat}{\widehat}
\newcommand{\ol}{\overline}

\newcommand{\wti}{\widetilde}

\newcommand{\wt}{\mathrm{wt}}

\newtheorem*{Thm2}{Theorem A}
\newtheorem*{Thm3}{Theorem B}

\newcommand{\bL}{\mathbf{L}}

\newcommand{\mA}{\mathcal{A}}

\renewcommand{\sl}{\mathfrak{sl}}
\newcommand{\Res}{\mathrm{Res}}
\newcommand{\mU}{\mathcal{U}}
\newcommand{\germ}{\mathfrak}
\newcommand{\het}{\mathrm{ht}}
\newcommand{\Ish}{I_{\mathrm{sh}}}
\newcommand{\mJ}{\mathcal{J}}

\newcommand{\mV}{\mathcal{V}}
\newcommand{\wh}{\widehat}

\title{Tensor products of Kirillov-Reshetikhin modules and fusion products}
\author{Katsuyuki Naoi}

\begin{document}

\address{%
Institute of Engineering \\
Tokyo University of Agriculture and Technology\\
2-24-16 Naka-cho, Koganei-shi, Tokyo 184-8588, JAPAN}
\email{naoik@cc.tuat.ac.jp}

%\subjclass[2010]{17B10}
%\keywords{current algebra, fusion product, Schur positivity}

\begin{abstract}
 We study the classical limit of a tensor product of Kirillov-Reshetikhin modules
 over a quantum loop algebra,
 and show that it is realized from the classical limits of the tensor factors using the notion of fusion products.
 In the process of the proof, we also give defining relations of the fusion product of the (graded) classical limits
 of Kirillov-Reshetikhin modules.
\end{abstract}

\maketitle

\section{Introduction}

Let $\fg$ be a complex simple Lie algebra, and $\bL\fg=\fg \otimes \C[t,t^{-1}]$ the associated loop algebra.
The theory of finite-dimensional representations over the quantum loop algebra $U_q(\bL\fg)$ has been intensively studied 
from various viewpoints over the past two decades.
\textit{Kirillov-Reshetikhin modules} (KR modules, for short) are a subclass of finite-dimensional simple $U_q(\bL\fg)$-modules,
and have attracted a particular interest
because of their rich combinatorial structures and several applications to mathematical physics (see 
\cite{MR906858,MR1745263,MR1993360,MR2254805,MR2403558,MR2642561,MR2767945,MR2998791} and references therein).

One approach which has been used to explore the structure of a finite-dimensional $U_q(\bL\fg)$-module
is to study its \textit{classical limit}, or \textit{graded limit}.
A classical limit is an $\bL\fg$-module which is obtained from a $U_q(\bL\fg)$-module by specializing the quantum parameter $q$ to $1$.
In addition, in many interesting cases, by restricting a classical limit to the current algebra $\fg[t] = \fg \otimes \C[t] \subseteq \bL\fg$ 
and taking a pull-back with respect to an automorphism, we obtain a graded $\fg[t]$-module, which is called the graded limit.
After the pioneering work of Chari and Pressley \cite{MR1850556}, several formulas of characters and multiplicities are obtained
for KR modules \cite{MR1836791,MR2238884} and their generalizations called minimal
affinizations \cite{MR2587436,MR3120578,MR3210588,LiNaoi}, by analyzing their classical or graded limits.

\begin{sloppypar}
Moreover, graded limits are important as well in view of the theory of finite-dimensional graded $\fg[t]$-modules,
since they provide  nontrivial and probably interesting such modules.
(Though the original motivation to study finite-dimensional graded $\fg[t]$-modules 
was mainly an application to the theory of $U_q(\bL\fg)$-modules,
they are now also of independent interest, since  they have connections with problems 
arising in mathematical physics such as the $X=M$ conjecture \cite{MR2290922,MR2428305,MR2964614},
and theory of symmetric functions such as Macdonald polynomials \cite{MR3371494}.)
In fact, in \cite{MR3296163} the authors have constructed a short exact sequence of $\fg[t]$-modules
as a graded limit analog of the $T$-system, which is a distinguished exact sequence of $U_q(\bL\fg)$-modules (see \cite{MR2254805}).
%a sequence of $U_q(\bL\fg)$-modules called the $T$-system.
This is an interesting example of the study of graded $\fg[t]$-modules, 
motivated by %the theory of $U_q(\bL\fg)$-modules through 
the notion of graded limits.
Since the process of obtaining a graded limit from a classical limit is elementary,
we will mainly focus on classical limits in this paper.
\end{sloppypar}

One difficulty in studying classical limits is the noncommutativity between the operations of tensor product and taking the limit.
Namely, the limit of a tensor product of $U_q(\bL\fg)$-modules is not necessarily isomorphic to the tensor product of their limits.
There are several examples (which will be listed below) which suggest that, to obtain the limit,
we have to replace (some of) tensor products with \textit{fusion products}.
Here fusion products are graded analogs of tensor products of graded $\fg[t]$-modules introduced by Feigin and Loktev in \cite{MR1729359}.
The purpose of this paper is to show the statement
for an arbitrary tensor product of KR modules.

Let $I$ be the index set of simple roots of $\fg$.
We denote a KR module by $W^{i,\ell}_q(a)$, which is parametrized by an index $i \in I$, a positive integer $\ell$ and a rational function
$a \in \C(q)$.
The graded limit of $W^{i,\ell}_q(a)$ is denoted by $W^{i,\ell}$ (which does not depend on $a$).
We now state the main theorem of this paper (Theorem \ref{Thm:Main}).

\begin{Thm2}
 Assume that a given tensor product $W^{i_1,\ell_1}_q(a_1)\otimes \cdots \otimes W^{i_p,\ell_p}_q(a_p)$ of KR modules has a classical limit
 (for the precise conditions, see Theorem \ref{Thm:Main}).\\
 {\normalfont(i)} If $a_1(1) = \cdots = a_p(1) = c \in \C^\times$,
 then the following isomorphism of $\fg[t]$-modules holds:
 \begin{equation*}\label{eq:isomorphism_intro}
  \ol{W^{i_1,\ell_1}_q(a_1) \otimes \cdots \otimes W^{i_p,\ell_p}_q(a_p)} \cong \varphi_c^*\big(W^{i_1,\ell_1} * \cdots * W^{i_p,\ell_p}\big).
 \end{equation*}
 Here the left-hand side is the classical limit, $*$ denotes the fusion product,
 and $\varphi_c^*$ the pull-back with respect to the automorphism $\varphi_c$ of $\fg[t]$ defined by 
 $\varphi_c\big(x\otimes f(t)\big)=x\otimes f(t+c)$.\\
 {\normalfont (ii)} In the general case, the following isomorphism of $\fg[t]$-modules holds:
  \[ \ol{W_q^{i_1,\ell_1}(a_1) \otimes\cdots \otimes W^{i_p,\ell_p}_q(a_p)} \cong \bigotimes_{c \in \C^\times} \varphi_c^* \,\Big(
     \underset{k;\, a_k(1) = c}{\text{\huge{\lower0.2ex\hbox{$*$}}}} W^{i_k,\ell_k}\Big),
  \]
  where {\huge{\lower0.2ex\hbox{$*$}}}\,$W^{i_k,\ell_k}$ denotes the fusion product of $W^{i_k,\ell_k}$'s.
% {\normalfont (ii)} In the general case, the following isomorphism of $\fg[t]$-modules holds:
% \[ \ol{W^{i_1,\ell_1}_q(a_1) \otimes \cdots \otimes W^{i_p,\ell_p}_q(a_p)} \cong \bigotimes_{c \in \C^\times} \varphi_c^*\big(
%    W^{i_1^{(c)},\ell_1^{(c)}} * \cdots * W^{i_{p_c}^{(c)},\ell_{p_c}^{(c)}}\big).
% \]
% Here $(i_1^{(c)},\ell_1^{(c)}),\ldots,(i_{p_c}^{(c)},\ell_{p_c}^{(c)})$ is a subsequence of $(i_1,\ell_1),\ldots,(i_p,
% \ell_p)$ defined by 
% \[ \big\{(i_1^{(c)},\ell_1^{(c)}),\ldots,(i_{p_c}^{(c)},\ell_{p_c}^{(c)})\big\} = \big\{(i_k,\ell_k)\bigm| 1\leq k \leq p \text{ such that }
%    a_k(1) = c\big\}
% \]
% as multisets.
\end{Thm2}

We remark that the assertion (i) implies that the graded limit of the tensor product is isomorphic to
the fusion product of their graded limits.
There are several special cases where the theorem has already been proved, which we list here.
\begin{itemize}
 \item In the case where $\ell_1=\cdots =\ell_p =1$, the result follows from \cite{MR2271991,MR2323538,MR2855081}.
 \item In the case of type $A$ with $i_1=\cdots=i_p$ and $a_1=\cdots =a_p$, the result is proved in \cite{BP}.
 \item For a special class of tensor products appearing in the $T$-system, the result follows from \cite{MR3296163}.
\end{itemize}

\noindent Our proof is valid for arbitrary $\fg$ and an arbitrary tensor product of KR modules,
even if it is reducible (as long as the classical limit exists).
It should be mentioned that a special class of tensor products in type $A$, whose factors are not necessarily isomorphic to KR 
modules, is treated in \cite{BCM}.

The proof of Theorem A is carried out in two steps. 
In the first step, we give defining relations of the fusion product of the graded limits of KR modules.
Let us mention the precise statement.
Let $\varpi_i$ ($i \in I$) be the fundamental weights.
Let $\fg = \fn_+ \oplus \fh \oplus \fn_-$ be a triangular decomposition,
and denote the Chevalley generators by $e_i,h_i,f_i$ ($i\in I$).
We show the following theorem (Theorem \ref{Thm:current}), which is a generalization of the result for $\fg=\sl_2$ 
in \cite{MR1988973,MR3296163}.

\begin{Thm3}
 A fusion product $W^{i_1,\ell_1}*\cdots*W^{i_p,\ell_p}$ is isomorphic to the $\fg[t]$-module
 generated by a single vector $v$ with relations
 \begin{equation*}
  \begin{split}
   \fn_+[t]v=0&, \ \ \big(h \otimes f(t)\big)v= f(0)\langle h,\gl\rangle v \ \text{ for } h \in \fh, f(t) \in \C[t],\\
   \Big(F_i(z)^r\Big)_s&v=0 \ \text{ for } i\in I, r >0, s<-\sum_{k;\, i_k = i} \min\{r, \ell_k\},
  \end{split}
 \end{equation*}
 where we set $\gl = \sum_{k=1}^p \ell_k\varpi_{i_k}$, and $\Big(F_i(z)^r\Big)_s \in U(\fg[t])$ denotes the coefficient of $z^s$ 
 in the $r$-th power of
 $F_i(z) = \sum_{k =0}^\infty (f_i\otimes t^k) z^{-k-1} \in U(\fg[t])[[z^{-1}]]$.
\end{Thm3}
%This theorem has previously been proved in \cite{MR1988973,MR3296163} for $\fg = \sl_2$.
%This is a generalization of the result in \cite{MR1988973,MR3296163} in which $\fg = \mathfrak{sl}_2$ is treated,
This theorem answers affirmatively the question raised in \cite[Introduction]{MR3226992} (see Remark \ref{Remark_for_current} (a)
of the present paper).
Note that the presentation is more refined than that given in \cite{Naoi:Schur} in type $A$.
The main tool we use for the proof is the functional realization of the dual of $U(\bL\fn_-)$,
which has been introduced in \cite{MR1275728} and further developed in \cite{MR1934307,MR2215613,MR2290922}.
In particular, in \cite{MR2215613,MR2290922} the authors also study fusion products of $W^{i,\ell}$'s using the functional
realization (though with a different motivation), which inspired our proof.

Using Theorem B, we can reduce the proof of Theorem A to the case $\fg = \mathfrak{sl}_2$
(the key to this is the fact that the relations in Theorem B essentially contain only root vectors corresponding to simple roots).
Then in the second step, we show the case of $\sl_2$ independently.
This is an outline of the proof of Theorem A.

The plan of this paper is as follows.
In Section \ref{Section:Preliminaries}, we give preliminary definitions and basic results.
In Section \ref{Section:Main_Thm}, we reduce the proof of Theorem A to the case $\fg = \sl_2$ and Theorem B.
In Section \ref{Section:Proof1} we prove Theorem B,
but we postpone some proofs of assertions concerning with the functional realization to Appendix \ref{Appendix}.
%since they might be already known to experts.
In Section \ref{section:proof_of_Uqsl2}, we show Theorem A for $\fg = \mathfrak{sl}_2$.
Finally in Appendix \ref{Appendix}, we give proofs postponed in Section \ref{Section:Proof1}.

\section{Preliminaries}\label{Section:Preliminaries}

\subsection{Lie algebras}\label{Subsection:Lie_algebras}

Let $\hat{C}=(c_{ij})_{0\leq i,j\leq n}$ be a Cartan matrix of nontwisted affine type,
and assume that  the indices are ordered as \cite[Section 4.8]{MR1104219}.
Let $C = (c_{ij})_{1\le i,j \le n}$, which is a Cartan matrix of finite type.
Set $I = \{1,\dots,n\}$ and $\hat{I} = \{0\} \cup I$.
Fix a diagonal matrix $\hat{D}=\mathrm{diag}(d_0,\ldots,d_n)$ such that $d_i\in \Z_{>0}$ and $\hat{D}\hat{C}$ is symmetric.

Let $\fg$ be the complex simple Lie algebra associated with $C$,
and fix a triangular decomposition $\fg = \fn_+ \oplus \fh \oplus \fn_-$. 
Let $\ga_i \in \fh^*$ be the simple roots, and $\varpi_i \in \fh^*$ the fundamental weights ($i \in I$).
Let $P$ and $Q$ be the weight and root lattices respectively, and set 
\[ P^+ = \sum_{i \in I} \Z_{\ge 0} \varpi_i, \ \ \ Q^+ =\sum_{i \in I} \Z_{\ge 0}\ga_i.
\]
For $\ggg = \sum_i m_i\ga_i \in Q^+$, let $\het(\ggg) = \sum_i m_i$ denote the \textit{height} of $\ggg$.
%Let $P^\vee= \bigoplus_{i\in I} \Z \varpi_i^\vee \subseteq \fh$ be the coweight lattice.
Denote by $R$ the root system, and by $R^+$ the set of positive roots.
For each root $\ga \in R$ denote by $h_\ga \in \fh$ its coroot.
Let $W$ be the Weyl group with simple reflections $\{s_i \mid i\in I\}$.
Let $\gt \in R^+$ be the highest root.
%, and $( \ , \ )$ the nondegenerate $W$-invariant symmetric bilinear form on $\fh^*$ such that $(\gt,\gt) =2$. 
%Denote by $\nu\colon \fh \to \fh^*$ the linear isomorphism defined by $(\nu(h), \gl) = \langle h, \gl\rangle$ for $h \in \fh$, $\gl \in \fh^*$.

Denote by $\fg_\ga$ ($\ga \in R$) the root spaces, and for each $\ga \in R^+$ fix vectors
$e_\ga\in \fg_\ga$ and $f_\ga\in \fg_{-\ga}$ satisfying $[e_\ga,f_\ga]=h_{\ga}$.
We also use the notations $e_i = e_{\ga_i}$, $f_i = f_{\ga_i}$, $h_i =h_{\ga_i}$.
For $i \in I$, denote by $\mathfrak{sl}_{2,i}$ the Lie subalgebra of $\fg$ spanned by $e_i,f_i,h_i$.
For $\gl \in P^+$, denote by $V(\gl)$ the finite-dimensional simple $\fg$-module with highest weight $\gl$.

Let $\varpi_i^\vee \in \fh$ ($i \in I$) be the fundamental coweights, and $P^\vee = \bigoplus_{i \in I} \Z \varpi_i^\vee \subseteq \fh$
the coweight lattice.
%Put
%\[ \varpi_i^\vee = \frac{2\varpi_i}{(\ga_i,\ga_i)}\in \fh^* \ \ \ \text{for $i \in I$},
%\]
%which satisfies $(\varpi_i^\vee, \ga_j) = \gd_{ij}$. 
%Set $P^\vee = \bigoplus_{i \in I} \Z \varpi_i^\vee \subseteq P$.
The group 
\[ \wti{W}= W \ltimes P^\vee
\]
is called the \textit{extended affine Weyl group}. 
Write $w$ for $(w,0) \in \wti{W}$, and $t_x$ for $(\id,x) \in \wti{W}$.
Let $\hat{Q} = Q \oplus \Z \gd$ be the affine root lattice, where $\gd$ is the null root.
$\wti{W}$ acts on $\hat{Q}$ by
\[ w(\gl + a\gd) = w(\gl) + a\gd \ \text{ and } \ t_x(\gl+ a\gd) = \gl + (a-\langle x,\gl\rangle) \gd,
\] 
where $w \in W$, $x \in P^\vee$, $\gl \in Q$, $a \in \Z$.
Set $\ga_0 = \gd - \gt \in \hat{Q}$ and $s_0 = s_\gt t_{-h_\gt}\in \wti{W}$,
where $s_\gt$ is the reflection with respect to $\gt$.
Let 
\[ \hW=\langle s_i \mid i \in \hI\,\rangle \subseteq \wti{W}
\]
be the affine Weyl group.
% which is the subgroup generated by $\{s_i\mid i\in \hI\}$.
Denote by $\gG$ the subgroup of $\wti{W}$ consisting of the elements preserving the set $\{\ga_i \mid i \in \hI\,\}$.
Then $\gG$ also acts on the set $\hI$ as permutations, and is identified with a subgroup of the Dynkin diagram automorphisms of $\hat{C}$.
It follows that $\wti{W} = \gG \ltimes \hat{W}$,
%There exists a certain subgroup $\gG$ of the Dynkin diagram automorphisms of $\hI$ such that $\wti{W} \cong 
%\gG \ltimes \hat{W}$, 
where $\tau \in \gG$ acts on $\hat{W}$ so that 
\[ \tau s_i \tau^{-1} = s_{\tau(i)} \ \text{ for $i \in \hI$}.
\]
%(see \cite{MR1890629}).
%$\hW$ is a Coxeter group with simple reflections $\{s_i \mid i \in \hI\}$, 
%and its length function $\ell\colon \hW \to \Z_{\geq 0}$ is extended to $\wti{W}$ 
%by setting $\ell(\tau w) = \ell(w)$ for $\tau \in \gG$ and $w \in \hat{W}$.

Given a complex Lie algebra $\fa$, its \textit{loop algebra} $\bL\fa$ is defined by the tensor product $\fa \otimes \C[t,t^{-1}]$,
whose Lie algebra structure is given by 
\[ [x\otimes f(t), y\otimes g(t)] = [x,y] \otimes f(t)g(t).
\]
Denote by $\fa[t]$ and $t^k\fa[t]$ for $k \in \Z_{>0}$ the Lie subalgebras $\fa \otimes \C[t]$ and $\fa \otimes t^k\C[t]$,
respectively.
The Lie algebra $\fa[t]$ is called the \textit{current algebra} associated with $\fa$.
For $c\in \C$, define a Lie algebra automorphism $\varphi_c = \varphi_c^{\fa}$ on $\bL\fa$ by
\[ \varphi_c\big(x \otimes f(t)\big) = x \otimes f(t+c) \ \text{ for } x \in \fa,\ f(t) \in \C[t].
\]

\subsection{Fusion product}\label{Subsection:Fusion product}

Here we recall the notion of fusion products introduced in \cite{MR1729359}.
Note that the degree grading on $\C[t]$ induces $\Z$-gradings on $\fg[t]$ and $U(\fg[t])$.
Denote by $\fg[t]^k$ and $U(\fg[t])^k$ the subspaces with degree $k$.
We say a $\fg[t]$-module $M$ is \textit{graded} if a $\Z$-grading $M= \bigoplus_{k \in \Z} M^k$ is given
and $\fg[t]^k M^\ell \subseteq M^{\ell+k}$ holds for all $k,\ell\in \Z$.

Let $M_1,\ldots,M_p$ be cyclic finite-dimensional graded $\fg[t]$-modules with respective generators $v_1,\ldots,v_p$,
and $c_1,\ldots,c_p$ pairwise distinct complex numbers.
For each $1\le j \le p$, define a $\fg[t]$-module $(M_j)_{c_j}$ by the pull-back $\varphi_{c_j}^*M_j$, 
and set $M = (M_1)_{c_1}\otimes \cdots \otimes (M_p)_{c_p}$.
It follows from \cite{MR1729359} that $M$ is generated by the single vector $v_1\otimes \cdots \otimes v_p$ as a $\fg[t]$-module.
The module $M$ is not graded, but a filtration is defined on $M$ by
\[ F^{\leq k}(M) = \sum_{\ell \leq k}U(\fg[t])^{\ell}(v_1\otimes \cdots \otimes v_p).
\]
%Obviously we have 
%\[ 0 =F^{\leq -1}(M) \subseteq F^{\leq 0}(M) \subseteq \cdots \subseteq F^{\leq N}(M) = M
%\]
%for sufficiently large $N$.
The associated graded space $\bigoplus_k F^{\leq k}(M)/F^{\leq k-1}(M)$ has a natural graded $\fg[t]$-module structure,
which is called the \textit{fusion product} of $M_1,\ldots,M_p$ and denoted by 
\[ M_1 * M_2*\cdots *M_p.
\]
Although the definition depends on the parameters $c_i$ and the generators $v_i$, we omit them for ease of notation. 
All fusion products appearing below are known not to depend on the parameters up to isomorphism,
and the choices of the generators will be clear from the context.
Note that the fusion product does not depend on the order of the factors up to isomorphism.

\subsection{Quantum loop algebras}\label{subsection:Quantum_loop_alg.}

%In this subsection we give a brief review on quantum loop algebras and their finite-dimensional modules.
%We refer the readers to \cite{MR1300632} for more details.

%Fix a positive integer $N$, and let $K = \C(q^{1/N})$.
Let $q$ be an indeterminate.
%$\C(q)$ be the field of rational functions in determinate $q$.
For $\ell \in \Z$, $s, s' \in \Z_{\ge 0}$ with $s\geq s'$, set 
\[ [\ell]_{q} = \frac{q^\ell - q^{-\ell}}{q-q^{-1}}, \ \ \ [s]_q! = [s]_q[s-1]_q\cdots[1]_q, \ \ \ \begin{bmatrix} s \\ s' \end{bmatrix}_q
   = \frac{[s]_q!}{[s']_q![s-s']_q!}.
\] 
Write $q_i = q^{d_i}$ for $i \in \hI$.
%Throughout this article, we fix a positive integer $d$.
The quantum loop algebra $U_q(\bL\fg)$ is a $\C(q)$-algebra generated by $k_i^{\pm 1}$, $X_i^{\pm}$ ($i\in \hI)$ with relations
\begin{equation*}\label{eq:def_rel_of_QLA}
 \begin{split}
  k_ik_i^{-1} = k_i^{-1}k_i = 1, \ \ k_ik_j = k_jk_i, \ \ k_iX_j^{\pm}k_i^{-1} = q_i^{\pm c_{ij}}X_j^\pm, \ \ k_\gd = 1,\\
  [X_i^+,X_j^-] = \gd_{ij}\frac{k_i -k_i^{-1}}{q_i-q_i^{-1}}, \ \ \sum_{k=0}^{s} (-1)^k \begin{bmatrix} s \\ k\end{bmatrix}_{q_i} 
  (X_i^{\pm})^kX_j^\pm(X_i^{\pm})^{s-k} = 0 \ (i \neq j),
 \end{split}
\end{equation*}
%where $r_i$ ($i \in \hI$) are the unique pairwise coprime positive integers satisfying $\sum_j c_{ij}r_j = 0$ for all $i$.
where $k_\gd = \prod_{i \in \hI} k_i^{a_i}$ with $\gd = \sum a_i\ga_i$, and $s = 1-c_{ij}$.
%We sometimes write $U_q(\bL\fg)_{\C(q^{1/d})}$ for $U_q(\bL\fg)$ in order to indicate the base field.
%The reason why we define $U_q(\bL\fg)$ over $K$ (not $\C(q)$) will be explained in Subsection \ref{Subsection:Subalgebras}.
Let $U_q(\fg)\subseteq U_q(\bL\fg)$ be the $\C(q)$-subalgebra generated by $k_i^{\pm 1},X_i^{\pm}$ ($i\in I$), which is the quantized enveloping
algebra associated with $\fg$.
The algebra $U_q(\bL\fg)$ has a Hopf algebra structure \cite{MR1227098,MR1300632}, and the comultiplication is given by
\[ \gD(X_i^+) = X_i^+ \otimes 1 + k_i \otimes X_i^+, \ \ \gD(X_i^-) = X_i^- \otimes k_i^{-1} + 1 \otimes X_i^-, \ \ 
   \gD(k_i) = k_i \otimes k_i.
\]
%For a positive integer $a$, let $U_{q^a}(\bL\fg)$ denote the $K$-algebra generated by the same symbols with $U_q(\bL\fg)$,
%subject to the relations (\ref{eq:def_rel_of_QLA}) with all $q$ replaced by $q^a$.

For $i \in \hI$, let $T_i = T_{i,1}''$ be the algebra automorphism of $U_q(\bL\fg)$ in \cite[Chapter 37]{MR1227098}.
Given $w = \tau w'\in \wti{W}$ with $\tau \in \gG$ and $w' \in \hW$, choose a reduced expression $w'=s_{i_1}\cdots s_{i_k}$ and
set $T_w= T_\tau T_{i_1}\cdots T_{i_k}$, where $T_\tau$ is the algebra automorphism on $U_q(\bL\fg)$ naturally induced from the diagram 
automorphism $\tau$.
The automorphism $T_w$ does not depend on the choice of the expression.
For $x \in P^\vee$, write $T_x = T_{t_{x}}$ for ease of notation.

%It follows from \cite{MR914215,MR1301623} that $U_q(\bL\fg)$ is isomorphic to the $\C(q)$-algebra with generators
There is another presentation of $U_q(\bL\fg)$ \cite{MR914215,MR1301623}.
In this presentation, $U_q(\bL\fg)$ is a $\C(q)$-algebra with generators   
\[ x_{i,r}^{\pm} \ (i \in I, r \in \Z), \ \ \ k_i^{\pm 1} \ ( i \in I), \ \ \ h_{i,m} \ (i \in I, m \in \Z \setminus \{0\})
\]  
and the following defining relations ($i,j \in I, r,r' \in \Z, m,m' \in \Z\setminus \{ 0\})$:
\begin{align*}
  k_ik_i^{-1} = k_i^{-1}&k_i =1, \ \ \ \ \ \ \ [k_i,k_j] = [k_i, h_{j,m}] = [h_{i,m},h_{j,m'}] = 0, \\
  k_i x_{j,r}^\pm k_i^{-1} = &\,q_i^{\pm c_{ij}}x_{j,r}^\pm,  \ \ \ \ \ \ \ 
  [h_{i,m}, x_{j,r}^{\pm}]= \pm \frac{1}{m}[m c_{ij}]_{q_i}x_{j,r+m}^{\pm}, \\
  &[x_{i,r}^{+}, x_{j,r'}^{-}] = \gd_{ij} \frac{\phi_{i,r+r'}^+ - \phi_{i,r+r'}^-}{q_i - q_i^{-1}},\\
  x_{i,r+1}^{\pm} x_{j,r'}^{\pm} -&\,q_i^{\pm c_{ij}}x_{j,r'}^\pm x_{i,r+1}^\pm 
  = q_i^{\pm c_{ij}}x_{i,r}^{\pm} x_{j,r'+1}^{\pm} - x_{j,r'+1}^{\pm}x_{i,r}^\pm,\\
  \sum_{\gs \in \mathfrak{S}_s} \sum_{k=0}^{s} (-1)^k \begin{bmatrix} s \\ k \end{bmatrix}_{\!q_i} 
    &x_{i, r_{\gs(1)}}^\pm \cdots x_{i, r_{\gs(k)}}^\pm x_{j,r'}^\pm x_{i,r_{\gs(k+1)}}^\pm \cdots x_{i,r_{\gs(s)}}^\pm
    = 0 \ \ (i \neq j)
\end{align*}
for all sequences of integers $r_1, \ldots,r_s$, where $s = 1 - c_{ij}$, $\mathfrak{S}_s$ is the symmetric group on $s$ letters,
and $\phi_{i,r}^\pm$'s are determined by equating coefficients of powers of $u$ in the formula
\[ \sum_{r = 0}^{\infty} \phi_{i, \pm r}^\pm u^{\pm r} = k_i^{\pm 1} \exp \left( \pm(q_i - q_i^{-1})\sum_{m = 1}^{\infty}
   h_{i, \pm m} u^{\pm m} \right),
\]
and $\phi_{i,\mp r}^{\pm} = 0$ for $r > 0$.
Let $o\colon I \to \{\pm 1\}$ be a map satisfying $o(i) = -o(j)$ whenever $c_{ij} < 0$.
Then we have 
\begin{equation}\label{eq:isom}
 x_{i,m}^{\pm} = o(i)^mT_{\varpi_i^\vee}^{\mp m}(X_i^{\pm}).
\end{equation}
Let $U_q(\bL\fn_{\pm})$ be the subalgebras of $U_q(\bL\fg)$ generated by $\{x_{i,r}^{\pm}\mid i\in I, r\in \Z\}$,
and $U_q(\bL\fh)$ the subalgebra generated by $\big\{k_i^{\pm 1}, h_{i,m}\bigm| i\in I, m\in \Z\setminus\{0\}\big\}$.
It is easily proved from the defining relations that 
\begin{equation}\label{eq:q-triangular}
  U_q(\bL\fg) = U_q(\bL\fn_-)U_q(\bL\fh)U_q(\bL\fn_+).
\end{equation}

%Let $U_q(\bL\mathfrak{sl}_2)$ be the quantum loop algebra associated with $\mathfrak{sl}_2$, with generators $K_i^{\pm 1}, X_i^{\pm}$ ($i\in
%\{0,1\}$. 
%By \cite[Proposition 3.8]{MR1301623}, there is a $K$-algebra isomorphism $\chi_i\colon U_{q_i}(\bL\mathfrak{sl}_2)
%\to U_q(\bL\fg)$ such that
%\begin{equation}\label{eq:sl2-isom}
% \begin{split}
%  \chi_i(k_1^{\pm 1}) &= k_i^{\pm 1},\ \ \chi_i(k_0^{\pm 1}) = k_i^{\mp 1},\ \
%  \chi_i(X_1^{\pm}) = x_{i,0}^{\pm}, \\ \chi_i(X_0^+) &= -o(i)k_i^{-1}x_{i,1}^-, \ \   \chi_i(X_0^-)=-o(i)x_{i,-1}^+k_i.
% \end{split}
%\end{equation}
%Denote by $U_{q,i}$ the image of $\chi_i$.

A $U_q(\fg)$-module $M$ is said to be \textit{of type $1$} if 
\[ M = \bigoplus_{\gl \in P} M_\gl, \ \ \ M_\gl = \big\{v \in M \bigm| k_i^{\pm 1}v = q_i^{\pm\langle h_i, \gl\rangle} v\big\}.
\]
In this paper, we will only consider $U_q(\fg)$-modules (and $U_q(\bL\fg)$-modules) of type $1$.
For $\gl \in P^+$, denote by $V_q(\gl)$ the simple $U_q(\fg)$-module (of type $1$) with highest weight $\gl$.

Let $P_q^+$ denote the monoid (under coordinate-wise multiplication) of $I$-tuples 
of polynomials $\bm{\pi}= \big(\bm{\pi}_1(u),\ldots,\bm{\pi}_n(u)\big)$
such that each $\bm{\pi}_i(u)$ is expressed as
\begin{equation}\label{eq:splitting_expression}
 \bm{\pi}_i(u) = (1 -a_{1}u)(1-a_{2}u) \cdots (1-a_{k}u)
\end{equation}
for some $k \ge 0$ and $a_{j} \in \C(q)^\times$.
Define a map $\wt\colon P^+_q \to P^+$ by 
\[ \wt(\bm{\pi}) = \sum_{i \in I} \big(\deg \bm{\pi}_i\big)\varpi_i.
\] 
%Let $\mathcal{P}^+$ be the monoid (under coordinatewise multiplication) consisting of $I$-tuples of polynomials 
%$\bm{\pi}=\big(\bm{\pi}_1(u),\ldots,\bm{\pi}_n(u)\big)$ with coefficients in 
%$K$ satisfying $\bm{\pi}_i(0)=1$ for all $i \in I$.
We say a $U_q(\bL\fg)$-module $V$ is \textit{$\ell$-highest weight} with \textit{$\ell$-highest weight vector} $v$ if it holds that
\[ x_{i,r}^+v = 0 \ \text{for} \ i\in I, r \in \Z, \ \ U_q(\bL\fh)v = \C(q)v, \ \ V = U_q(\bL\fg)v.
\]
It follows from \cite{MR1357195} that for every $\bm{\pi} \in P_q^+$, there exists a unique (up to isomorphism) 
simple finite-dimensional $\ell$-highest weight $U_q(\bL\fg)$-module, which we denote by $L_q(\bm{\pi})$, such that 
its $\ell$-highest weight vector $v_{\bm{\pi}}$ (which is unique up to a scalar multiplication) satisfies
\begin{equation}\label{eq:l-highest}
 k_i^{\pm 1} v_{\bm{\pi}} = q_i^{\pm \langle h_i,\wt(\bm{\pi})\rangle}v_{\bm{\pi}}, \ \ h_{i,m}v_{\bm{\pi}} = [m]_{q_i}d_{i,m}v_{\bm{\pi}}
 \ \text{ for $i \in I, m \in \Z\setminus \{0\}$},
\end{equation}
where $d_{i,m} \in \C(q)$ are determined by the formula
\begin{equation*} 
 \mathrm{exp}\bigg(-\sum_{m=1}^\infty d_{i,\pm m}u^m\bigg) = \bm{\pi}_i^{\pm}(u).
\end{equation*}
Here $\bm{\pi}_i^+(u) = \bm{\pi}_i(u)$, and $\bm{\pi}_i^-(u)= (1-a_1^{-1}u)\cdots (1-a^{-1}_ku)$ if (\ref{eq:splitting_expression}) holds.
%$u^{\deg \pi_i} \bm{\pi}_i(u^{-1})\big/ \big(u^{\deg \pi_i}\bm{\pi}_i(u^{-1})\big)\big|_{u=0}$.
If $V$ is an $\ell$-highest weight module and its $\ell$-highest weight vector
$v_{\bm{\pi}}$ satisfies (\ref{eq:l-highest}), we say that the $\ell$-highest weight of $V$ is $\bm{\pi}$.
The following lemma is a consequence of \cite[Corollary 6.9]{MR1810773}.

\begin{Lem}\label{Lem:dual}
 There exists a bijection $I \ni i \mapsto \bar{i} \in I$ and an integer $K$ such that,
 if we set $\bm{\pi}^* = \big(\bm{\pi}_{\bar{i}}(q^{K}u)\big)_{i \in I}$ for $\bm{\pi} \in P_q^+$,
 then the dual module $L_q(\bm{\pi})^*$ is isomorphic to $L_q(\bm{\pi}^*)$.
\end{Lem}

In this paper, we are mainly interested in the following special simple modules.

\begin{Def}\normalfont
 Given $i \in I, \ell \in \Z_{> 0}$ and $a \in \C(q)$, define $\bm{\pi}_{i,\ell,a} \in P^+_q$ by 
 \[ (\bm{\pi}_{i,\ell,a})_j(u) = \begin{cases}
                  \prod_{k=0}^{\ell-1} (1-aq_i^{2k}u) & \text{if} \ j=i,\\
                  1               & \text{otherwise.}
                 \end{cases}
 \]
 The simple $U_q(\bL\fg)$-module $L_q(\bm{\pi}_{i,\ell,a})$ is called the \textit{Kirillov-Reshetikhin module} 
 (KR module for short) associated with $i,\ell,a$, and denoted by $W^{i,\ell}_q(a)$.
\end{Def}
%Note that $W_q^{i,0}(a)$ is a trivial module for all $i,a$.

We end this subsection by recalling the following lemma.

\begin{Lem}\label{Lem:commutativity_of_tensor}
 {\normalfont(i)} For $\bm{\pi},\bm{\pi}' \in P^+_q$, the simple $U_q(\bL\fg)$-module $L_q(\bm{\pi}\bm{\pi}')$ is isomorphic 
  to a quotient of  the $U_q(\bL\fg)$-submodule of $L_q(\bm{\pi}) \otimes L_q(\bm{\pi}')$ generated by the tensor product
  of $\ell$-highest weight vectors. In particular if $L_q(\bm{\pi})\otimes L_q(\bm{\pi}')$ is simple, then we have
  \[ L_q(\bm{\pi}) \otimes L_q(\bm{\pi}') \cong L_q(\bm{\pi}') \otimes L_q(\bm{\pi})
  \]
  as $U_q(\bL\fg)$-modules. \\
 {\normalfont(ii)} Let $W_q^{i_1,\ell_1}(a_1)$ and $W_q^{i_2,\ell_2}(a_2)$ be two KR modules, and assume that 
 $a_1 \notin q^{\Z}a_2$.
 Then we have
 \[ W_q^{i_1,\ell_1}(a_1) \otimes W_q^{i_2,\ell_2}(a_2) \cong W_q^{i_2,\ell_2}(a_2) \otimes W_q^{i_1,\ell_1}(a_1).
 \]
\end{Lem}

\begin{proof}
 The assertion (i) is proved in \cite{MR1300632, MR1357195}.
 It follows from \cite[Theorem 6.1]{MR1883181} and Lemma \ref{Lem:dual} that the module 
 $W_q^{i_1,\ell_1}(a_1) \otimes W_q^{i_2,\ell_2}(a_2)$ in (ii)
 and its dual are both $\ell$-highest weight, and hence simple.
 Now the assertion (ii) follows from (i), and the proof is complete.
\end{proof}
%Set $\varpi_{a}(u) = 1-au \in K[u]$ for $a \in K$,
%\[ \big(\bpi_{i,a}\big)_j(u) = \begin{cases} 1-au & i = j \\
%                                             1    & \text{otherwise},
%                               \end{cases} 
%\]
%and $\pi_{\ell,a}(u) \in K[u]$ for $\ell \in \Z_{> 0}$ by 
%\[ \pi_{\ell,a}(u) = \prod_{k=0}^{\ell-1} \varpi_{aq_i^{2k}}(u).
%\]
%Let $\mathcal{P}_\Z^+ \subseteq\mathcal{P}^+$ be the submonoid generated by $\bpi_{i,a}$ ($i\in I$, $a \in q^{\Z}$).

\subsection{Subalgebras \boldmath{$U_{q,i}$}}\label{Subsection:Subalgebras}

For $i \in I$, let $U_{q,i}$ be the $\C(q)$-subalgebra of $U_q(\bL\fg)$ generated by $x_{i,r}^{\pm}$ ($r\in \Z$) and 
$k_i^{\pm 1}$.
Denote by $U_{q}(\bL\mathfrak{sl}_2)_{d_i}$ the quantum loop algebra associated with $\hat{C}= \begin{pmatrix} 2 & -2 \\ -2 & 2
\end{pmatrix}$ and $\hat{D} = \begin{pmatrix} d_i & 0 \\ 0 & d_i \end{pmatrix}$.
%Let $\psi_i$ be the $\C$-algebra isomorphism from $\C(q^{1/d_id})$ to $\C(q^{1/d})$
%defined by $\psi_i(q^{1/d_iN}) =q^{1/N}$.
By \cite[Proposition 3.8]{MR1301623}, there is a $\C(q)$-algebra isomorphism $\Psi_i\colon U_{q}(\bL\mathfrak{sl}_2)_{d_i}
\stackrel{\sim}{\to} U_{q,i}$ such that
\begin{equation}\label{eq:sl2-isom}
 \begin{split}
  \Psi_i(k_1^{\pm 1}) &= k_i^{\pm 1},\ \ \Psi_i(k_0^{\pm 1}) = k_i^{\mp 1},\ \ \Psi_i(X_1^{\pm}) = x_{i,0}^{\pm},\\
   \Psi_i(X_0^+) &= -o(i)k_i^{-1}x_{i,1}^-, \ \   \Psi_i(X_0^-)=-o(i)x_{i,-1}^+k_i.
 \end{split}
\end{equation}
It should be remarked that $U_{q,i}$ is not a sub-coalgebra, and hence $\Psi_i$ is not a coalgebra isomorphism.
The following lemma is needed later.
%For the proof, see \cite{MR1402568}.

\begin{Lem}\label{Lem:restriction}
 Let $\bm{\pi}^1,\ldots, \bm{\pi}^p$ be a sequence of elements of $P_q^+$, and assume that $L_q(\bm{\pi}^1) \otimes
 \cdots \otimes L_q(\bm{\pi}^p)$ is $\ell$-highest weight.
 Then the $U_{q,i}$-submodule 
 \begin{equation*}
  U_{q,i}(v_{\bm{\pi}^1} \otimes \cdots \otimes v_{\bm{\pi}^p}) \subseteq L_q(\bm{\pi}^1) \otimes \cdots \otimes L_q(\bm{\pi}^p)
 \end{equation*}
 generated by the tensor product of $\ell$-highest weight vectors is isomorphic, as a $U_{q}(\bL\mathfrak{sl}_2)_{d_i}$-module, to 
  \begin{equation}\label{eq:Uqi-mod2}
   L_q\big(\bm{\pi}^1_i(u)\big) \otimes \cdots \otimes L_{q}\big(\bm{\pi}^p_i(u)\big).
  \end{equation}
\end{Lem}

\begin{proof}
 Let $\gl = \sum_{k=1}^p \wt(\bm{\pi}^k) \in P^+$.
 Since $L_q(\bm{\pi}^1) \otimes \cdots \otimes L_q(\bm{\pi}^p)$ is $\ell$-highest weight, we easily see from (\ref{eq:q-triangular}) and
 the weight consideration that
 \begin{align}\label{eq:Uqi-mod1}
  U_{q,i}(v_{\bm{\pi}^1} \otimes \cdots \otimes v_{\bm{\pi}^p}) &= \bigoplus_{k \in \Z_{\ge 0}}
    \Big(L_q(\bm{\pi}^1)\otimes \cdots \otimes L_q(\bm{\pi}^p)\Big)_{\gl - k\ga_i}\nonumber\\ 
  &= \left(U_{q,i}v_{\bm{\pi}^1}\right) \otimes \cdots \otimes \left(U_{q,i}v_{\bm{\pi}^p}\right).
 \end{align}
  By \cite[Lemma 2.3]{MR1402568}, each $U_{q,i}v_{\bm{\pi}^k}$ is isomorphic to $L_{q}\big(\bm{\pi}^k_i(u)\big)$ as a 
  $U_q(\bL\mathfrak{sl}_2)_{d_i}$-module.
  Note that, the module (\ref{eq:Uqi-mod2}) is defined through the coproduct of $U_{q}(\bL\mathfrak{sl}_2)_{d_i}$,
  and hence it is not obvious that this is isomorphic to the module in (\ref{eq:Uqi-mod1}),
  which is defined through the coproduct of $U_q(\bL\fg)$. 
  However this is proved in [loc.\ cit., Lemma 2.2], and hence the lemma follows.
\end{proof}

\subsection{Classical limits}\label{Subsection:classical_limits}

Let $\mA$ be the local subring of $\C(q)$ defined by
\[ \mA = \left\{ f/g \bigm| f, g \in \C[q],\ g(1) \neq 0\right\}.
\]
An \textit{$\mA$-lattice} $L$ of a $\C(q)$-vector space $V$ is a free $\mA$-submodule such that 
$V \cong \C(q) \otimes_{\mA} L$. 
Let $U_{\mathcal{A}}(\bL\fg) \subseteq U_q(\bL\fg)$ be the $\mA$-subalgebra generated by $k_i^{\pm 1}, x_{i,r}^\pm$ ($i\in I, r \in \Z$), 
and define $U_{\mA}(\bL\fn_\pm)\subseteq U_q(\bL\fn_\pm)$, $U_{\mA}(\bL\fh)\subseteq U_q(\bL\fh)$
and $U_{\mA,i}\subseteq U_{q,i}$ ($i\in I$) by the $\mA$-subalgebras generated by the given generators of the respective 
$\C(q)$-algebras.
It is easily seen from the defining relations that $h_{i,m} \in U_{\mA}(\bL\fg)$,
and then it is proved that
\begin{equation}\label{eq:A-triangular}
 U_{\mA}(\bL\fg) = U_{\mA}(\bL\fn_-)U_{\mA}(\bL\fh)U_{\mA}(\bL\fn_+). 
\end{equation}
%By \cite{MR2066942} $U_{\mA}(\bL\fg)$ coincides with the $\mA$-subalgebra generated by $k_i^{\pm 1},x_i^{\pm}$ ($i\in \hI$).
%Hence it follows from \cite{MR1227098} that $U_{\mA}(\bL\fg)$ is an $\mA$-lattice and a sub-coalgebra of $U_q(\bL\fg)$.  
%The following lemma is proved in \cite{MR1712630,MR1836791} in simply-laced or classical types.

\begin{Lem}\label{Lem:A-form}
 The $\mA$-subalgebra $U_{\mA}(\bL\fg)$ coincides with the $\mA$-subalgebra generated by $k_i^{\pm 1}, X_i^{\pm}$ {\normalfont($i\in \hI$)}.
 In particular, $U_{\mA}(\bL\fg)$ is an $\mA$-lattice \cite{MR1066560} and a sub-coalgebra.
\end{Lem}

\begin{proof}
 Let $\wti{U}_{\mA}(\bL\fg)$ denote the $\mA$-subalgebra generated by $k_i^{\pm 1}, X_i^{\pm}$ ($i \in \hI$).
 Since $T_w^{\pm 1}$ ($w \in \wti{W}$) preserve $\widetilde{U}_{\mA}(\bL\fg)$, the containment
 $U_{\mA}(\bL\fg) \subseteq \wti{U}_{\mA}(\bL\fg)$ follows from (\ref{eq:isom}).
 To show the opposite containment, we shall prove first that $T_w^{\pm 1}$ ($w \in \wti{W}$) also preserve $U_{\mA}(\bL\fg)$,
 which is equivalent to that $T_{\varpi_i^{\vee}}^{\pm 1}$ 
 and $T_i^{\pm 1}$ ($i\in I$) preserve $U_{\mA}(\bL\fg)$ by \cite[Lemma 2.8]{MR991016}. 
 The former is easily proved from the fact 
 %$T_{\varpi_i^\vee}T_{\varpi_j^\vee} = T_{\varpi_j^\vee}T_{\varpi_i^\vee}$ \cite[1.4.(g)]{MR991016} and 
 $T_{\varpi_i^\vee}(x_{j,r}^{\pm}) = x_{j,r}^{\pm}$ ($i \neq j$) \cite[Corollary 3.2]{MR1301623}.
 To show the latter, we need to prove that $T_i^{\pm 1}(x_{j,r}^+) \in U_\mA(\bL\fg)$ and $T_i^{\pm 1}(x_{j,r}^-) \in U_{\mA}(\bL\fg)$ for all 
 $i, j \in I$ and $r\in\Z$.
 We prove the first assertion (the other is similarly proved).
 If $i \neq j$, then 
 \[ T_i^{\pm 1} (x_{j,r}^+) = o(j)^rT_i^{\pm 1}T_{\varpi_j^\vee}^{-r}(X_j^+) = o(j)^rT_{\varpi_j^\vee}^{-r}T_i^{\pm 1}(X_j^+) 
 \]
 (see \cite[Lemma 2.2]{MR991016}), which belongs to $U_{\mA}(\bL\fg)$ since $T_i^{\pm 1}(X_j^+)$
 is contained in the $\mA$-subalgebra generated by $k_a^{\pm 1}$ and $X_a^\pm$ ($a \in I$). 
 Assume that $i=j$, and consider the isomorphism $\Psi_i\colon U_{q}(\bL\mathfrak{sl}_2)_{d_i} \to U_{q,i}$. 
 By \cite[Corollary 3.8]{MR1301623}, we have $\Psi_i \circ T_1 = T_i \circ \Psi_i$ and $\Psi_i \circ T_{\varpi_1^\vee} = T_{\varpi_i^\vee}\circ
 \Psi_i$. 
 Hence it follows that 
 \[ T_i^{\pm 1}(x_{i,r}^+) = o(i)^r\Psi_i\!\left(T_1^{\pm 1}T_{\varpi_1^\vee}^{-r} (X_1^+)\right) \in \Psi_i\!\left(
    \wti{U}_\mA(\bL\mathfrak{sl}_2)_{d_i}\right).
 \]
 It is easily seen from (\ref{eq:sl2-isom}) that $\Psi_i$ maps $\wti{U}_\mA(\bL\mathfrak{sl}_2)_{d_i}$ into $U_\mA(\bL\fg)$,
 and hence we have $T_i^{\pm 1}(x_{i,r}^+) \in U_\mA(\bL\fg)$, as required.

 Let $w \in \wti{W}$ and $i \in I$ be such that $w(\ga_i) = \ga_0$. 
 By \cite{MR1227098}, we have $T_w(X_i^\pm) = X_0^\pm$,
 and hence $X_0^\pm \in U_\mA(\bL\fg)$ follows from the assertion proved above. 
 This implies $U_{\mA}(\bL\fg) \supseteq \wti{U}_{\mA}(\bL\fg)$, and the proof is complete.
\end{proof}

%For $\bm{\pi} \in \mP^+$, let $L_{\mA}(\bm{\pi})$ denote the $U_\mA(\bL\fg)$-submodule 
%$U_{\mA}(\bL\fg)v_{\bm{\pi}} \subseteq L_q(\bm{\pi})$.
Let $P^+_{\mA}$ be the submonoid of $P^+_q$ consisting of $\bm{\pi} = \big(\bm{\pi}_1(u),\ldots,\bm{\pi}_n(u)\big)$ such that
$\bm{\pi}_i^\pm(u) \in \mA[u]$ for all $i \in I$.
%which is equivalent to that $\pi_i(u) \in \mA[u]$ and the coefficient of the highest degree belongs to 
%$\mA^\times$ for all $i\in I$. 
The following lemma in simply-laced or classical types follows from \cite{MR1850556,MR1836791},
and the proof can be extended to general types.
For completeness we give a more elementary proof here
(in the papers cited above the assertion is proved over $\C[q,q^{-1}]$, and therefore the proof is more involved).

\begin{Lem}\label{Lem:A-lattice}
 Assume that $V$ is a finite-dimensional $\ell$-highest weight $U_q(\bL\fg)$-module 
 with $\ell$-highest weight $\bm{\pi} \in P^+_{\mA}$, and let $v_{\bm{\pi}}$ be an $\ell$-highest weight vector. 
 Then the $U_\mA(\bL\fg)$-submodule $L = U_{\mA}(\bL\fg)v_{\bm{\pi}} \subseteq V$ is an $\mA$-lattice of $V$.
\end{Lem}

\begin{proof}
 By (\ref{eq:l-highest}) and (\ref{eq:A-triangular}), we have $L = U_{\mA}(\bL\fn_-)v_{\bm{\pi}}$.
 Let $N = \max\{\deg \bm{\pi}_i(u)\mid i \in I\}$. 
 We first show the claim that any vector of the form
 \begin{equation}\label{eq:vector_of_the_form}
  x_{i_1,k_1}^-x_{i_2,k_2}^-\cdots x_{i_p,k_p}^-v_{\bm{\pi}} 
 \end{equation}
 can be written as an $\mA$-linear combination of vectors
 \[ x_{i_1',k_1'}^-x_{i_2',k_2'}^-\cdots x_{i_p',k_p'}^-v_{\bm{\pi}}, \ \ \ |k_j'| < N + p.
 \]
% where $N(\eta)$ depends only on $\eta =\sum_{1 \le j \le p} \ga_{i_j}$ and $\bm{\pi}$.
 We proceed by the induction on $p$. The case $p=1$ follows from \cite[Proposition 4.3]{MR1850556}.
 Let $p >1$, and assume that $k_1 \geq 0$ (the case $k_1 < 0$ is similarly proved).
 By the induction hypothesis we may assume that $|k_j| < N + p-1$ for $2\leq j \leq p$. Then using the relation
 \[ x_{i_1,k_1}^-x_{i_{2},k_2}^- = q_{i_1}^{-c_{i_1i_2}}x_{i_{2},k_2}^-x_{i_1,k_1}^- + q_{i_1}^{-c_{i_1i_2}}x_{i_1,k_1-1}^-x_{i_{2},k_2+1}^--
    x_{i_{2},k_2+1}^-x_{i_1,k_1-1}^-,
 \]
 and applying the induction hypothesis on $p$ again, we easily see that the claim is proved by the induction on $k_1$.

 Since $V$ is finite-dimensional, vectors of the form (\ref{eq:vector_of_the_form}) are $0$ when $p$ is sufficiently large.
 Hence it follows from the claim that $L$ is finitely generated as an $\mA$-module.
 Since $L$ is obviously torsion-free, the lemma follows.
\end{proof}

Denote by $\overline{\phantom{q} }\colon \mA \to \C$ the $\C$-algebra homomorphism defined by $\ol{q} = 1$.
Given an $\mA$-module $M$, denote by $M_\C$ the $\C$-vector space $\C \otimes_{\mA} M$ where $\mA$ acts on $\C$ via $\ol{\phantom{q}}$.
For an element $v \in M$, write $\ol{v} = 1 \otimes v \in M_\C$. 
$U_{\mA}(\bL\fg)_\C$ has a natural $\C$-algebra structure, and there exists a surjective $\C$-algebra homomorphism 
(see \cite[Proposition 9.2.3]{MR1300632})
\[ U_\mA(\bL\fg)_\C \to U(\bL\fg)\colon \ol{x_{i,r}^+} \mapsto e_i \otimes t^r, \ \ \ol{x_{i,r}^-} \mapsto f_i \otimes t^r, \ \ 
   \ol{k_i^{\pm 1}} \mapsto 1,\ \ \ol{h_{i,m}} \mapsto h_i \otimes t^m,
\]
whose kernel is the ideal of $U_{\mA}(\bL\fg)_\C$ generated by $\ol{k_i}-1$ ($i \in I$).
Assume that $V$ is a finite-dimensional $\ell$-highest weight $U_q(\bL\fg)$-module with $\ell$-highest weight $\bm{\pi} \in P_{\mA}^+$
and $\ell$-highest weight vector $v_{\bm{\pi}}$, and set $L = U_{\mA}(\bL\fg)v_{\bm{\pi}} \subseteq V$.
Then through the algebra homomorphism $L_\C$ becomes an $\bL\fg$-module,
which is called the \textit{classical limit} of $V$ and denoted by $\ol{V}$.
%If we further assume that $V$ is finite-dimensional and its $\ell$-highest weight belongs to $P^+_\mA$, 
By Lemma \ref{Lem:A-lattice} we have $\dim_{\C(q)} V = \dim_\C \ol{V}$,
and it is also easy to see that 
\[ \Big[ V : V_q(\gl) \Big]_{U_q(\fg)} = \Big[ \ol{V} : V(\gl) \Big]_{\fg}
\]
holds for all $\gl \in P^+$ (see \cite[Subsection 3.4]{MR3120578}, for example), 
where the left- and right-hand sides denote the multiplicities as $U_q(\fg)$- and $\fg$-modules, respectively.

The following $\fg[t]$-modules are introduced in \cite{MR1836791,MR2238884}.

\begin{Def}\normalfont
 For $i \in I$, $\ell \in \Z_{> 0}$ and $c\in \C$, let $W^{i,\ell}(c)$ be a $\fg[t]$-module generated by a vector $v=v_{i,\ell,c}$ 
 with relations
 \begin{equation}\label{eq:def_rel}
  \begin{split}
   \fn_+[t]v=&\,0, \ \ \big(h\otimes f(t)\big) v = \ell f(c)\langle h, \varpi_i\rangle v \ \text{for} \ h\in \fh, \ 
   f(t) \in \C[t],\\ 
   f_i^{\ell+1}v&=\big(f_i\otimes (t-c)\big)v =0,\ \ f_jv=0 \ \text{for} \ j \in I \setminus\{i\}.
  \end{split}
 \end{equation}
When $c = 0$, we simply write $W^{i,\ell}=W^{i,\ell}(c)$.
\end{Def}
Note that $\varphi_c^*\big(W^{i,\ell}(c')\big) = W^{i,\ell}(c+c')$.
It is easily seen that, for sufficiently large $N$, $\fg \otimes (t-c)^N\C[t]$ acts trivially on $W^{i,\ell}(c)$.
Hence when $c \neq 0$, by considering the Lie algebra homomorphism 
\begin{equation}\label{eq:Laurent_extension}
 \bL\fg \twoheadrightarrow \fg \otimes \left(\C[[t-c]]/(t-c)^N\C[[t-c]]\right) \cong \fg \otimes \left(\C[t]/(t-c)^N\C[t]\right)
\end{equation}
induced by the Taylor expansion, 
$W^{i,\ell}(c)$ is uniquely lifted to an $\bL\fg$-module.
%It is also easy to check that the defining relation as a $\bL\fg$-module is given by 
%\[ \bL\fn_+v_{i,\ell}=0, \ h \otimes f(t) v_{i,\ell} = \ell f(a)\langle h, \varpi_i\rangle v_{i,\ell} \ \text{for}\ h \in \fh, \ f(t) \in 
%   \C[t,t^{-1}],
%\]
%and (\ref{eq:def_rel2}).

\begin{Prop}\label{Prop:classical_limit_of_W}
 Let $i \in I, \ell \in \Z_{> 0}$ and $a \in \mA^\times$.
 The classical limit $\ol{W^{i,\ell}_q(a)}$ of the KR module $W^{i,\ell}_q(a)$ is isomorphic to $W^{i,\ell}(\ol{a})$ as an $\bL\fg$-module.
\end{Prop}

\begin{proof}
 It is easy to check that the relations (\ref{eq:def_rel}) with $\C[t]$ replaced by $\C[t,t^{-1}]$ are the defining relations of $W^{i,\ell}(c)$
 ($c \neq 0$) as an $\bL\fg$-module.
 Then the existence of a surjection $W^{i,\ell}(\ol{a}) \twoheadrightarrow \ol{W^{i,\ell}_q(a)}$ 
 is proved as in \cite[Lemmas 2.3 and 2.4]{MR1836791}. Hence it suffices to show that $\dim_{\C(q)} W^{i,\ell}_q(a) = \dim_\C 
 W^{i,\ell}(\ol{a})$.
 This is proved in \cite{MR1836791} for classical types,
 and deduced in general types as the special case of a single tensor factor of \cite[Corollary 5.1]{MR2767945}.
\end{proof}

By the proposition, the pull-back $\varphi_{-\ol{a}}^*\Big(\,\ol{W_q^{i,\ell}(a)}\,\Big)$ is isomorphic to $W^{i,\ell}$,
which has a graded $\fg[t]$-module structure.
In this reason, $\varphi_{-\ol{a}}^*\Big(\,\ol{W_q^{i,\ell}(a)}\,\Big)$ is called the \textit{graded limit} of $W_q^{i,\ell}(a)$.

We end this subsection with recalling a theorem in the case $\fg = \mathfrak{sl}_2$.
%Let $\mathrm{ev}_c\colon \mathfrak{sl}_2[t] \to \mathfrak{sl}_2$ ($c \in \C$) be the Lie algebra homomorphism (called evaluation map)
%defined by $\mathrm{ev}_c\big(x \otimes f(t)\big) = f(c)x$,
%and $V(\ell)$ $(\ell \in \Z_{\ge 0})$ the $(\ell+1)$-dimensional simple $\mathfrak{sl}_2$-module.
In this case $I = \{1\}$ is a singleton set, and therefore we write $e,h,f$ for $e_1,h_1,f_1$, $W^\ell(c)$ for $W^{1,\ell}(c)$,
and $W^\ell$ for $W^{1,\ell}$.
%We also write $W^\ell(c)$ (resp.\ $W^\ell$) for $W^{1,\ell}(c)$ (resp.\ $W^{1,\ell}$).
Note that $W^{\ell}(c)$ is just the pull-back of the $(\ell+1)$-dimensional simple $\sl_2$-module with respect to the evaluation map at $t=c$:
\[ \sl_2[t] \to \sl_2 \colon x \otimes f(t) \mapsto f(c)x \ \ \ \text{for} \ x \in \sl_2, \ f(t) \in \C[t].
\]
%the restriction of $W^{\ell}(c)$ to $\mathfrak{sl}_2$ is an $(\ell+1)$-dimensional simple module.
The following is one of the main results in \cite{MR1988973} (see also \cite[Section 6]{MR3296163}).

\begin{Thm}\label{Thm:MR1988973}%[{\cite[Proposition 2.4]{MR1988973}}]
 Assume that $\fg = \sl_2$, 
 and define a power series $F(z) \in U(\sl_2[t])[[z^{-1}]]$ in an indeterminate $z$ by
 \[ F(z) = \sum_{k =0}^\infty (f \otimes t^k) z^{-k-1}.
 \]
 For a sequence $\ell_1,\ldots,\ell_p$ of positive integers, 
 the fusion product $W^{\ell_1} * \cdots * W^{\ell_p}$
 is isomorphic to the $\mathfrak{sl}_2[t]$-module generated by a vector $v$ with relations
 \begin{equation*}\label{eq:Feigin2}
  \begin{split}
  (e \otimes \C[t]) v=0, \ \ \big(h&\otimes f(t)\big)v = (\ell_1+\cdots+\ell_p) f(0) v \ \ \text{for} \ f(t) \in \C[t],\\
  \Big(F(z)^r\Big)_s v &= 0 \ \text{ for }  r >0, s<-\sum_{k= 1}^p \min\{r, \ell_k\},
  \end{split}
 \end{equation*}
 where $\Big(F(z)^r\Big)_s \in U(\sl_2[t])$ denotes the coefficient of $z^s$ in the $r$-th power of $F(z)$.
\end{Thm}

%Later we will generalize this theorem to a general $\fg$.
%The following is obvious from the defining relations.
%
%
%\begin{Lem}\label{Lem:elem._sl2}
%  {\normalfont(i)} Let $j \in I$.
%   The $\mathfrak{sl}_{2}^{(j)}[t]$-submodule of $W^{i,\ell}(c)$ generated by $v_{i,\ell,c}$ is isomorphic, as a $\mathfrak{sl}_2[t]$-module, to 
%   $\mathrm{ev}_c^*V(\ell)$ if $i=j$, and a trivial module if $i\neq j$.\\
%  {\normalfont(ii)} When $\fg = \mathfrak{sl}_2$,
%   the module $W^{\ell}(c)$ {\normalfont(}$\ell \in \Z_{\ge 0}${\normalfont)} is isomorphic to $\mathrm{ev}_c^*V(\ell)$.
%\end{Lem}

\section{Main Theorem}\label{Section:Main_Thm}

\subsection{The statement of the main theorem}

%For $c \in \C^\times$, let $\mA_c^\times$ denote the subset of $\mA^\times$ consisting of $a \in \mA^\times$ satisfying $a(1) = c$.
The following is the main theorem of this paper.

\begin{Thm}\label{Thm:Main}
 Let $i_1,\ldots,i_p \in I$, $\ell_1,\ldots,\ell_p \in \Z_{>0}$, and $a_1,\ldots,a_p \in \mA^\times$,
 and assume that the tensor product $W_q^{i_1,\ell_1}(a_1) \otimes\cdots \otimes W^{i_p,\ell_p}_q(a_p)$
 is $\ell$-highest weight.\\
 {\normalfont(i)} If $a_1(1) = \cdots = a_p(1) = c \in \C^\times$,
  then the classical limit of the tensor product is isomorphic 
  as a $\fg[t]$-module to the pull-back with respect to $\varphi_c$ of the fusion product of $W^{i_k,\ell_k}$ 
  {\normalfont($1\le k \le p$)}, i.e.,
  \[ \ol{W_q^{i_1,\ell_1}(a_1) \otimes\cdots \otimes W^{i_p,\ell_p}_q(a_p)} \cong \varphi^*_c\big(W^{i_1,\ell_1} * \cdots 
     *W^{i_p,\ell_p}\big).
  \]
 {\normalfont(ii)} In the general case, define a finite subset $\mathcal{C} = \{ a_k(1) \mid 1\leq k \leq p\} \subseteq \C^\times$.
  Then the following isomorphism of $\fg[t]$-modules holds:
  \[ \ol{W_q^{i_1,\ell_1}(a_1) \otimes\cdots \otimes W^{i_p,\ell_p}_q(a_p)} \cong \bigotimes_{c \in \mathcal{C}} \varphi_c^* \,\Big(
     \underset{k;\, a_k(1) = c}{\text{\huge{\lower0.2ex\hbox{$*$}}}} W^{i_k,\ell_k}\Big),
  \]
  where {\huge{\lower0.2ex\hbox{$*$}}}\,$W^{i_k,\ell_k}$ denotes the fusion product of $W^{i_k,\ell_k}$'s.
\end{Thm}

\begin{Rem}\normalfont
 (a) As far as the author knows, no necessary and sufficient conditions are known 
  for a tensor product of KR modules to be $\ell$-highest weight, but a sufficient condition has been obtained
  in \cite{MR1883181}.\\ 
 (b) In the setting of the assertion (i), the graded $\fg[t]$-module 
  \[ \varphi_{-c}^*\Big(\,\ol{W_q^{i_1,\ell_1}(a_1) \otimes\cdots \otimes W^{i_p,\ell_p}_q(a_p)}\,\Big)
  \]
  is called the graded limit of $W_q^{i_1,\ell_1}(a_1) \otimes\cdots \otimes 
  W^{i_p,\ell_p}_q(a_p)$.
  In this terminology, the assertion (i) claims that the graded limit of the tensor product is isomorphic to the fusion product of 
  the graded limits.\\
% Denote by $W$ the $\fg[t]$-module $\varphi_c^*\left(W^{i_1,\ell_1} * \cdots *W^{i_p,\ell_p}\right)$ in (i).
 (c) The $\fg[t]$-module $\varphi_c^*\left(W^{i_1,\ell_1} * \cdots *W^{i_p,\ell_p}\right)$ in the assertion (i)
 is uniquely lifted to an $\bL\fg$-module via the Lie algebra homomorphism  (\ref{eq:Laurent_extension}),
 and then the isomorphism in (i) becomes that of $\bL\fg$-modules.
 The isomorphism in (ii) is also lifted to that of $\bL\fg$-modules in the same way.
\end{Rem}

\subsection{Reduction to \boldmath{$U_q(\bL\mathfrak{sl}_2)$} case}

Theorem \ref{Thm:Main} is proved by reducing it to the case of $U_q(\bL\mathfrak{sl}_2)$,
and for the reduction we need to prove another theorem which gives defining relations of the fusion product of $W^{i,\ell}$'s 
(Theorem \ref{Thm:current}).
%from two assertions stated below, one concerned with $\fg[t]$-modules (Theorem \ref{}) and the other 
%with $U_q(\bL\mathfrak{sl}_2)$-modules (Proposition \ref{}).
In this subsection we will present the statement of Theorem \ref{Thm:current}, and then show that Theorem \ref{Thm:Main} is indeed reduced to 
the $U_q(\bL\mathfrak{sl}_2)$ case (Proposition \ref{Prop:Uqsl2}) via this theorem.
The proofs of Theorem \ref{Thm:current} and Proposition \ref{Prop:Uqsl2} will be given in the next two sections.

Define for each $i \in I$ a power series $F_i(z) \in U(\fg[t])[[z^{-1}]]$ by
\[ F_i(z) = \sum_{k =0}^\infty (f_i \otimes t^k) z^{-k-1}.
\]
For any formal series $f(z)$ in $z$, denote by $f(z)_s$ ($s \in \Z$) the coefficient of $z^s$.
%Recall that a partition is a nonincreasing sequence of nonnegative integers $\bm{\ell} = (\ell_1 \geq \ell_2 \geq \cdots)$
%such that $\ell_N = 0$ for sufficiently large $N$. 
%Let $|\bm{\ell}| = \sum_k \ell_k$ be the size of $\bm{\ell}$ and $l(\bm{\ell}) = \max\{k \mid \ell_k \neq 0\}$ the length of $\bm{\ell}$.
In Section \ref{Section:Proof1}, we will show the following theorem, which is a generalization of Theorem \ref{Thm:MR1988973}.

\begin{Thm}\label{Thm:current}
 Suppose that sequences $i_1,\ldots,i_p$ of elements of $I$ and $\ell_1,\ldots,\ell_p$ of positive integers are given.
 Set $\gl = \sum_{k=1}^p \ell_k \varpi_{i_k}$, and define a subset 
 \[ S_i =\{1 \leq k \leq p \mid i_k = i\}
 \]
 for each $i \in I$.
 Then the fusion product $W^{i_1,\ell_1} * \cdots * W^{i_p,\ell_p}$ is isomorphic to the $\fg[t]$-module generated 
 by a single vector $v$ with relations
 \begin{equation}\label{eq:Defining_rel.of_W(ell)}
  \begin{split}
  \fn_+[t]v&=0, \ \ \big(h \otimes f(t)\big)v= f(0)\langle h, \gl\rangle v \ \text{ for } h \in \fh, f(t) \in \C[t],\\
  \Big(F_i(z)^r\Big)_sv&=0 \ \text{ for } i\in I, r >0, s<-\sum_{k \in S_i} \min\{r, \ell_k\}.
  \end{split}
 \end{equation}
\end{Thm}

\begin{Rem}\label{Remark_for_current}\normalfont
% We think this theorem is of interest in its own right, since it has several connections with previous results.
% Let us list a few of them.\\
 (a) For an $I$-tuple of partitions $\bm{\mu}= (\mu^{(1)},\ldots,\mu^{(n)})$ with 
  $\mu^{(i)} = (\mu_1^{(i)}\geq \ldots\geq \mu_{p_i}^{(i)})$, denote by $W(\bm{\mu})$ the fusion product 
  ${\text{{\LARGE{\lower0.2ex\hbox{$*$}}}}}_{i,k}W^{i,\mu_k^{(i)}}$.
%  $\{W^{i,\mu^{(i)}_k}\mid i\in I, 1 \le k \le p_i\}$.
  Then Theorem \ref{Thm:current} implies the following fact: if two $I$-tuples of partitions $\bm{\mu}=(\mu^{(1)},\ldots,\mu^{(n)})$ and 
  $\bm{\nu}=(\nu^{(1)},\ldots,\nu^{(n)})$ satisfy $|\mu^{(i)}| = |\nu^{(i)}|$ and $\mu^{(i)} \leq \nu^{(i)}$ (with respect to 
  the dominance order) for all $i \in I$, 
  then there exists a surjective $\fg[t]$-module homomorphism $W(\bm{\mu}) \twoheadrightarrow W(\bm{\nu})$.
  Indeed for every $i \in I$ and $r >0$, setting $k_r= \max\{\, k \mid \mu_k^{(i)} > r\}$, the assumption implies that
  \[ \sum_{k \ge 1} \min\{r,\mu_k^{(i)}\} = rk_r + \sum_{k > k_r}\mu^{(i)}_k \geq rk_r+ \sum_{k > k_r}\nu^{(i)}_k \geq 
     \sum_{k \geq 1} \min\{r,\nu_k^{(i)}\},
  \]
  and hence a surjection $W(\bm{\mu}) \twoheadrightarrow W(\bm{\nu})$ exists by Theorem \ref{Thm:current}.
  This surjection can be viewed as an extension of the Schur positivity of KR-modules proved in \cite{MR3226992} to the current 
  algebra setting (see also \cite{Naoi:Schur}).\\
 (b) In \cite{MR3296163}, the authors have introduced a collection of $\fg[t]$-modules $V(\bm{\xi})$ indexed by an 
  $R^+$-tuple of partitions
  $\bm{\xi} = (\xi^{(\ga)})_{\ga \in R^+}$ satisfying $|\xi^{(\ga)}| = \langle h_\ga, \mu\rangle$ for some $\mu \in P^+$.
  In their terminology, Theorem \ref{Thm:current} says that the module $W^{i_1,\ell_1} * \cdots * W^{i_p,\ell_p}$
  is isomorphic to $V(\bm{\xi})$, 
  where $\bm{\xi} = (\xi^{(\ga)})_{\ga \in R^+}$ is defined by 
   \[ \xi^{(\ga)} = \begin{cases} \mathrm{part}\,\{\ell_k \mid k\in S_i\} & \text{if} \ \ga = \ga_i,\\
                                  \big(1^{\langle h_\ga,\gl\rangle}\big) & 
                                       \text{if} \ \ga \ \text{is not simple}.
                    \end{cases}
   \]
   Here $\mathrm{part}\,T$ for a multiset of positive integers $T$ denotes the partition obtained by ordering 
   the elements of $T$.
%  (The generator $v$ automatically satisfies the relations of $V(\bm{\xi})$ for non-simple $\ga$, see.)
%  (The relations for non-simple $\ga$ automatically follow, see \cite[Proposition 3.4]{MR3296163}.)\\
\end{Rem}

%In Section \ref{section:proof_of_Uqsl2}, we will prove the following proposition.
We will prove the following proposition in Section \ref{section:proof_of_Uqsl2}.
Though this is just a special case of Theorem \ref{Thm:Main} (i) for $\fg = \mathfrak{sl}_2$,
we write the precise statement here for later reference
(we write $W_q^{\ell}(a)$ for $W_q^{1,\ell}(a)$ and $W^{\ell}$ for $W^{1,\ell}$).

\begin{Prop}\label{Prop:Uqsl2}
 Assume that $\fg = \mathfrak{sl}_2$. 
 Let $\ell_1,\ldots,\ell_p \in \Z_{>0}$ and $a_1,\ldots,a_p \in \mathcal{A}^\times$ be sequences such that 
 $a_1(1) = \cdots =a_p(1) = c \in \C^\times$ and $W_q^{\ell_1}(a_1) \otimes \cdots \otimes W_q^{\ell_p}(a_p)$ is $\ell$-highest weight.
 Then we have
 \[ \ol{W_q^{\ell_1}(a_1) \otimes \cdots \otimes W_q^{\ell_p}(a_p)} \cong \varphi_c^*\big(W^{\ell_1}*\cdots *W^{\ell_p}\big)
 \]
 as $\sl_2[t]$-modules.
% Theorem \ref{Thm:Main} {\normalfont(i)} is true when $\fg = \mathfrak{sl}_2$. 
\end{Prop}

Now we deduce Theorem \ref{Thm:Main} in full generality, assuming Theorem \ref{Thm:current} and Proposition \ref{Prop:Uqsl2}.
First we show the assertion (i). 
Let $v_k$ ($1\leq k\leq p$) denote an $\ell$-highest weight vector of $W_q^{i_k,\ell_k}(a_k)$.
In this proof we use the following abbreviations:
\[ W_q = W_q^{i_1,\ell_1}(a_1)\otimes \cdots \otimes W_q^{i_p,\ell_p}(a_p), \ \ \ v=v_1 \otimes \cdots \otimes v_p \in W_q.
\]
Note that $v$ is an $\ell$-highest weight vector of $W_q$ by the assumption.
Fix $i \in I$ for a moment, and
let $k_1,\ldots,k_{p_i}$ be the subsequence of $1,\ldots,p$ such that 
\[ \{k_1,\ldots,k_{p_i}\} = \{1\leq k \leq p \mid i_k = i\}.
\]
%Since $W_q$ is generated by $v$, we easily see from (\ref{eq:q-triangular}) and the weight consideration that
%\[ W_q \supseteq U_{q,i}v = \bigoplus_{\mu \in \gl-\Z_{\ge 0} \ga_i} \big(W_q\big)_\mu = U_{q,i}v_1 \otimes \cdots \otimes U_{q,i}v_p,
%\]
%and then it follows from 
By Lemma \ref{Lem:restriction}, the $U_{q,i}$-submodule $U_{q,i}v$ of $W_q$ is isomorphic, as a $U_{q}(\bL\mathfrak{sl}_2)_{d_i}$-module,
to 
\[ W_q^{\ell_{k_1}}(a_{k_1}) \otimes \cdots \otimes W_q^{\ell_{k_{p_i}}}(a_{k_{p_i}}).
\]
Obviously this is an $\ell$-highest weight $U_{q}(\bL\mathfrak{sl}_2)_{d_i}$-module.
Now we consider the classical limit $\ol{U_{q,i}v} = \C \otimes_\mA U_{\mA,i}v$, which is an $\bL\mathfrak{sl}_{2,i}$-module.
By Proposition \ref{Prop:Uqsl2}, $\ol{U_{q,i}v}$ is isomorphic to
$\varphi_{c}^*\big(W^{\ell_{k_1}} * \cdots * W^{\ell_{k_{p_i}}}\big)$ as an $\mathfrak{sl}_2[t]$-module, 
and hence it holds that
\[ \varphi_{-c}^*\ol{U_{q,i}v} \cong W^{\ell_{k_1}} * \cdots * W^{\ell_{k_{p_i}}}.
\]
%Let $\bm{\ell}' = (\ell_1',\ldots,\ell_{p_i}')$ be the partition obtained by reordering $\{\ell_{j_1},\ldots,\ell_{j_{p_i}}\}$.
%such that
%\[ \{\ell_1',\ldots,\ell_r'\} = \{\ell_{j_1},\ldots,\ell_{j_r}\}
%\]
%holds as multisets.
Then by Theorem \ref{Thm:MR1988973}, the vector 
$\bar{v} = 1\otimes v \in \varphi_{-c}^*\ol{U_{q,i}v}$ satisfies the relations
\begin{equation}\label{eq:relations_for_i}
 \begin{split}
  \big(e_i\otimes \C[t]\big)\bar{v}&=0, \ \ \big(h_i \otimes f(t)\big)\bar{v}= f(0)\langle h_i,\gl\rangle
  \bar{v} \ \text{ for } f(t) \in \C[t],\\ 
  \Big(F_i(z)^r\Big)_s\bar{v}&=0 \ \text{ for } r >0, s<-\sum_{j= 1}^{p_i} \min\{r, \ell_{k_j}\},
 \end{split}
\end{equation}
where $\gl = \sum_{k=1}^{p} \ell_{k}\varpi_{i_k}$.
Moreover there is an $\mathfrak{sl}_{2,i}[t]$-module homomorphism $\ol{U_{q,i}v} \to \ol{W_q}$ since $U_{\mA,i} v \subseteq U_\mA(\bL\fg)v$,
which induces a homomorphism
$\varphi_{-c}^*\ol{U_{q,i}v} \to \varphi_{-c}^*\ol{W_q}$.
Hence $\bar{v} \in \varphi_{-c}^*\ol{W_q}$ also satisfies the relations (\ref{eq:relations_for_i}).
Now applying this argument to all $i \in I$,  it follows from Theorem \ref{Thm:current} that there exists a surjective $\fg[t]$-module 
homomorphism 
\[ W^{i_1,\ell_1} * \cdots * W^{i_p,\ell_p} \twoheadrightarrow \varphi_{-c}^* \ol{W_q}.
\]
It follows from Proposition \ref{Prop:classical_limit_of_W} 
that the dimensions of these modules are equal. 
Hence this is an isomorphism and the assertion (i) is proved.

Now the assertion (ii) of Theorem \ref{Thm:Main} is deduced from (i) as follows.
For each $c \in \mathcal{C}$, set 
\[ W_q^{(c)} = \bigotimes_{k;\, a_k(1)=c} W^{i_k,\ell_k}_q(a_k),
\]
where the factors are ordered so that $W^{i_r,\ell_r}_q(a_r)$ is left to $W_q^{i_s,\ell_s}(a_s)$ if $r<s$.
It follows from Lemma \ref{Lem:commutativity_of_tensor} (ii) that
\[ W_q^{i_1,\ell_1}\otimes \cdots \otimes W_q^{i_p,\ell_p} \cong \bigotimes_{c \in \mathcal{C}} W_q^{(c)},
\]
where the order of the right-hand side is arbitrary.
Hence each $W_q^{(c)}$ is $\ell$-highest weight since so is $W_q^{i_1,\ell} \otimes \cdots \otimes W_q^{i_p,\ell_p}$.
Let $v^{(c)}$ ($c \in \mathcal{C}$) be an $\ell$-highest weight vector of $W_q^{(c)}$.
Since $U_\mA(\bL\fg)$ is a sub-coalgebra, we have
\[ U_{\mA}(\bL\fg)\bigotimes_{c \in \mathcal{C}} v^{(c)} \subseteq \bigotimes_{c\in \mathcal{C}} U_{\mA}(\bL\fg)v^{(c)}
\]
in $\bigotimes_{c} W_q^{(c)}$, and hence we obtain by (i) an $\bL\fg$-module homomorphism
\[ \ol{\bigotimes_{c \in \mathcal{C}} W_q^{(c)}} \to \bigotimes_{c \in \mathcal{C}} \ol{ W_q^{(c)}} \cong \bigotimes_{c \in \mathcal{C}} 
   \varphi_c^*\,\Big(\underset{k;\, a_k(1) = c}{\text{\huge{\lower0.2ex\hbox{$*$}}}} W^{i_k,\ell_k}\Big).
%   \varphi_{c_1}^*W^{c_1} \otimes \cdots \otimes \varphi_{c_s}^*W^{c_s}
\]
By \cite[Proposition 1.4]{MR1729359} the module $\bigotimes_{c} \varphi_{c}^*\big(${\LARGE{\lower0.2ex\hbox{$*$}}}\,$W^{i_k,\ell_k}\big)$
%\varphi_{c_1}^*W^{c_1} \otimes \cdots \otimes \varphi_{c_s}^*W^{c_s}$
is generated by the image of $\ol{\bigotimes_{c}v^{(c)}}$,
and hence the homomorphism is surjective.
Then the comparison of the dimensions shows that this is an isomorphism, and (ii) is proved.

\section{Proof of Theorem \ref{Thm:current}}\label{Section:Proof1}

\subsection{Reduction to a study of a quotient space of \boldmath{$U(\bL\fn_-)$}}\label{Subsection:reduction}

This section is devoted to the proof of Theorem \ref{Thm:current}.
Fix sequences $i_1,\ldots,i_p$ of elements of $I$ and $\ell_1,\ldots,\ell_p$ of positive integers.
As in the theorem, set $\gl = \sum_k \ell_k \varpi_{i_k}$ and $S_i = \{1\leq k \leq p \mid i_k = i\}$.

Write $W = W^{i_1,\ell_1} * \cdots * W^{i_p,\ell_p}$, and let $M$ denote the $\fg[t]$-module generated by a vector $v$ with
relations (\ref{eq:Defining_rel.of_W(ell)}).
%Denote by $v_W$ and $v_M$ the canonical generators of $W$ and $M$, respectively.
We begin with the following lemma.

\begin{Lem}\label{eq:exsitence_of_surjection}
 There exists a surjective $\fg[t]$-module homomorphism from $M$ to $W$.
\end{Lem}

\begin{proof}
 Let $v_k$ ($1\leq k \leq p$) denote the generator of $W^{i_k,\ell_k}$.
 For $i \in I$, it is clear from the defining relations that the $\sl_{2,i}[t]$-submodule $U(\sl_{2,i}[t]) v_k\subseteq W^{i_k,\ell_k}$ 
 is isomorphic to 
 $W^{\ell_k}$ if $k \in S_i$, and a trivial module otherwise.
 Hence, letting $c_1,\ldots,c_p$ be pairwise distinct complex numbers, there is an injective $\mathfrak{sl}_{2,i}[t]$-module homomorphism
 \[ \bigotimes_{k \in S_i} \varphi_{c_k}^*W^{\ell_k} \hookrightarrow \bigotimes_{k=1}^p \varphi_{c_k}^*W^{i,\ell_k}.
 \]
 By the definition of the fusion product, this induces, for every $i \in I$, an $\sl_{2,i}[t]$-module homomorphism from the fusion product 
 ${\text{{\LARGE{\lower0.2ex\hbox{$*$}}}}}_{k\in S_i}W^{\ell_k}$
% of $\{W^{\ell_k} \mid k \in S_i\}$
 to $W$. 
 Then it follows from Theorem \ref{Thm:MR1988973} that the generator of $W$ satisfies the relations 
 (\ref{eq:Defining_rel.of_W(ell)}).
 %Hence it follows from Theorem \ref{Thm:MR1988973} that the generator of $W$ satisfies the relations (\ref{eq:Feigin2}) 
 %(with obvious modifications such as replacing $e$ with $e_i$, etc.).
 %Since this holds for every $i \in I$, 
 Hence there is a $\fg[t]$-module homomorphism from $M$ to $W$, which is obviously surjective.
 The lemma is proved.
\end{proof}

%For $\mu \in P^+$, denote by $V(\mu)$ the simple $\fg$-module with highest weight $\mu$.
%Given a finite-dimensional $\fg$-module $V$, we denote by $\Big[V : V(\mu)\Big]$ the multiplicity of $V(\mu)$ in $V$.
%We note the following useful equality:
%\[ \Big[V:V(\mu)\Big] = \dim \Big(V / \fn_-V\Big)_\mu,
%\]
%which follows from the finite-dimensional representation theory of $\fg$.
%Here $V/\fn_-V$ is $P$-graded by the action of $\fh$.

Since both $M$ and $W$ are finite-dimensional $\fg$-modules, Lemma \ref{eq:exsitence_of_surjection}
implies that, for the proof of Theorem \ref{Thm:current}, it suffices to show the inequalities 
\[ \Big[M : V(\mu)\Big] \leq \Big[ W : V(\mu)\Big]
\]
of multiplicities
%$\Big[M : V(\mu)\Big] \leq \Big[ W : V(\mu)\Big]$ 
as $\fg$-modules for all $\mu \in P^+$.
% for all $\mu \in P^+$.
To show this, we will give an upper bound of $\Big[M :V(\mu)\Big]$ below.

To begin with, note that the following equality follows from the finite-dimensional representation theory of $\fg$:
\begin{equation}\label{eq:inequality_of_M}
 \Big[M:V(\mu)\Big] = \dim \Big(M / \fn_-M\Big)_\mu.
\end{equation}
Here $M/\fn_-M$ is $P$-graded by the action of $\fh$.

To shorten the notation we write $U^-$ for $U(\fn_-[t])$ from now on.
Let $\mathcal{I}$ be the left $U^-$-ideal generated by the vectors
\[ \Big\{\Big(F_i(z)^r\Big)_s \Bigm| i\in I, r >0, s<-\sum_{k \in S_i} \min\{r, \ell_k\}\Big\}.
\]
The $\fn_-[t]$-module $U^-/\mathcal{I}$ is naturally graded by $-Q^+$,
%and we denote the weight space by $\Big(U(\fn_-[t])/\mathcal{I}\Big)_{-\ga}$ ($\ga \in Q^+$).
and obviously there exists a surjective $\fn_-[t]$-module homomorphism from $U^-/\mathcal{I}$ to $M$,
which maps $\Big(U^-/\mathcal{I}\Big)_{-\ggg}$ for $\ggg \in Q^+$ onto $M_{\gl-\ggg}$.
Then this homomorphism yields the following surjective linear map
\[ \Big(U^-/\big(\fn_-U^- + \mathcal{I}\big)\Big)_{-\ggg} \twoheadrightarrow \Big(M/\fn_-M\Big)_{\gl -\ggg},
\]
which implies the inequality
\begin{equation}\label{eq:inequality_of_dim}
 \dim\Big(U^-/\big(\fn_-U^- + \mathcal{I}\big)\Big)_{-\ggg} \geq \Big[M : V(\gl-\ggg)\Big]
\end{equation}
by (\ref{eq:inequality_of_M}).

Next we will define a quotient space of $U(\bL\fn_-)$ which is linearly isomorphic to $U^-/\big(\fn_-U^- + \mathcal{I}\big)$.
In the sequel we write $U^-_{\bL} = U(\bL\fn_-)$ for ease of notation.
Fix a sufficiently large positive integer $N$,
and let $\wti{F}_i(z) \in U^-_{\bL}((z))$ be the formal Laurent series defined by
\[ \wti{F}_i(z) = \sum_{k=-\infty}^N (f_i \otimes t^k) z^{-k-1}.
\]
%Note that for every $r > 0$ and $s \in \Z$, $\Big(\wti{F}_i(z)^r\Big)_s$ acts on $\ol{1} \in \ol{U}_-^{\bL}$
%since all but finitely many terms vanishes.
Denote by $\mathcal{J}$ the left $U^-_{\bL}$-ideal generated by the vectors
\[ \{f_i \otimes t^k \mid i \in I, k > N\} \cup\Big\{\,\Big(\wti{F}_i(z)^r\Big)_s \Bigm| i \in I, r >0, 
   s<-\sum_{k \in S_i} \min\{r, \ell_k\}\Big\}.
\]
%and $\mathcal{J} \subseteq U^-_{\bL}$ the left ideal obtained by taking the inverse image of $\ol{\mathcal{J}}$
%with respect to the canonical surjection $U^-_{\bL} \twoheadrightarrow \ol{U}_-^{\bL}$.
The definition of $\mathcal{J}$ does not depend on the choice of $N$,
since 
\[ \wti{F}_i(z)_{-k-1} = f_i \otimes t^k \in \mathcal{J} \ \text{ for } \#\,S_i \leq k \leq N.
\] %since
%\[ f_i \otimes t^k = \Big(\wti{F}_i(z)\Big)_{-k-1} \in \mathcal{J} \ \ \ \text{if} \ \# S_i < k \leq N.
%\]

%\begin{Rem}\normalfont
% Roughly speaking, the ideal $\mathcal{J}$ is the one generated by 
% \[ \Big\{\,\Big(\wti{F}_i(z)^r\Big)_s \,\Big|\, i \in I, r >0, s<-\sum_{k \in S_i} \min\{r, \ell_k\}\Big\},
% \]
% but we have defined it as above to avoid infinite sums.
%\end{Rem}

\begin{Lem}\label{Lem:coincidence_of_ideals}
 {\normalfont(i)} The left ideal $U^-_{\bL}\,\mathcal{I}$ coincides with $\mathcal{J}$.\\
 {\normalfont(ii)} There exists a linear isomorphism 
  \[ U^-_{\bL}\big/\big(\fn_-[t^{-1}]U^-_{\bL} + \mathcal{J} \big) \stackrel{\sim}{\to} U^-\big/\big(\fn_-U^-  +\mathcal{I}\big)
  \]
  preserving their $-Q^+$-gradings.
\end{Lem}

\begin{proof}
 (i) We show $U^-_{\bL}\, \mathcal{I} \subseteq \mathcal{J}$ (the other containment is similarly proved).
  For $i \in I$, set 
  \[ Z_i = \big\{(r,s)\bigm| r>0, s < -\sum_{k \in S_i} \min\{r,\ell_k\}\big\} \subseteq \Z_{>0} \times \Z,
  \]
  and $F_i'(z) = \sum_{k=0}^N (f_i \otimes t^k)z^{-k-1}$.
  Since $f_i \otimes t^k \in \mathcal{J}$ for $k>N$, it suffices to show that $\Big(F_i'(z)^r\Big)_s \in \mathcal{J}$ if $(r,s) \in Z_i$.
  We show by the induction on $k$ that
  \[ \Big(F_i'(z)^{k}\wti{F}_i(z)^{r-k}\Big)_s \in \mathcal{J} \ \ \text{if} \ \ (r,s) \in Z_i,
  \] 
  which with $k = r$ completes the proof.
  For $k=0$ there is nothing to prove.  When $k >0$, we have
  \begin{align*}
   \Big(F_i'&(z)^k\wti{F}_i(z)^{r-k}\Big)_s \\
   &= \Big(F_i'(z)^{k-1}\wti{F}_i(z)^{r-k+1} - \sum_{a< 0}z^{-a-1}(f_i\otimes t^{a})
     F_i'(z)^{k-1}\wti{F}_i(z)^{r-k}\Big)_s \\
   &= \Big(F_i'(z)^{k-1}\wti{F}_i(z)^{r-k+1}\Big)_s - \sum_{a<0}f_i\otimes t^{a}\Big(F_i'(z)^{k-1}\wti{F}_i(z)^{r-k}
    \Big)_{s+a+1}.
  \end{align*}
  Since $(r,s) \in Z_i$ implies $(r-1, s+a+1) \in Z_i$ for all $a < 0$, the right-hand side belongs to $\mathcal{J}$ by the induction 
  hypothesis. The proof is complete.\\
%  , and hence $U^-_{\bL}\, \mathcal{I} \subseteq \mathcal{J}$ is proved.\\
 (ii) By the Poincar\'{e}-Birkhoff-Witt theorem, we have 
   \[ U^-_{\bL} = (t^{-1}\fn_-[t^{-1}])U^-_{\bL} \oplus U^-.
   \]
   Consider the composition 
   \[ U^-_{\bL} \twoheadrightarrow U^- \twoheadrightarrow U^-\big/\big(\fn_-U^-  +\mathcal{I}\big),
   \]
   where the first map is the projection with respect to the above decomposition, and the second is the canonical one.
   Obviously the kernel of this linear map is 
   \[ (t^{-1}\fn_-[t^{-1}])U^-_{\bL} + \fn_- U^- + \mathcal{I},
   \]
   which we denote by $\mathcal{K}$.
   It suffices to show that
   \begin{equation*}\label{eq:equality_of_ideals}
    \mathcal{K} = \fn_-[t^{-1}]U^-_{\bL} + \mathcal{J},
   \end{equation*}
   and the containment $\subseteq$ is clear from (i).
   We show the other containment.
   Since $\ad(\fn_-)(t^{-1}\fn_-[t^{-1}]) \subseteq t^{-1}\fn_-[t^{-1}]$, 
   we have 
   \[ \fn_-U^-_{\bL} = \fn_-\Big((t^{-1}\fn_-[t^{-1}])U^-_{\bL} \oplus U^-\Big) \subseteq (t^{-}\fn_-[t^{-1}])U^-_{\bL} + \fn_-U^- \subseteq
      \mathcal{K},
   \] 
   and hence $\fn_-[t^{-1}]U^-_{\bL} \subseteq \mathcal{K}$ holds. 
   On the other hand, it also follows from (i) that
   \[ \mathcal{J} = U^-_{\bL} \, \mathcal{I} = \big((t^{-1}\fn_-[t^{-1}])U^-_{\bL} \oplus U^-\big)\mathcal{I}
      \subseteq (t^{-1}\fn_-[t^{-1}])U^-_{\bL} + \mathcal{I} \subseteq \mathcal{K}.
   \]
   Therefore $\mathcal{K} \supseteq \fn_-[t^{-1}]U^-_{\bL} + \mathcal{J}$ follows, as required.
%   The equality (\ref{eq:equality_of_ideals}) is proved.
\end{proof}

In conclusion, the proof of Theorem \ref{Thm:current} is now reduced to the following proposition by (\ref{eq:inequality_of_dim})
and Lemma \ref{Lem:coincidence_of_ideals} (ii), which will be proved in Subsection \ref{subsection:current_complete}.

\begin{Prop}\label{Prop:critical_inequality}
% Let $\gl = \sum_{k=1}^p \ell_k \varpi_{i_k}$.
 For each $\ggg \in Q^+$ such that $\gl - \ggg \in P^+$, the inequality
 \[ \dim\Big(U^-_{\bL}\big/\big(\fn_-[t^{-1}]U^-_{\bL} + \mathcal{J} \big)\Big)_{-\ggg} \leq \Big[W : V(\gl - \ggg)\Big]
 \]
 holds.
\end{Prop}

\subsection{Functional realization of the dual space of \boldmath{$U^-_{\bL}$}}
%\subsection{Proof of Proposition \ref{Prop:critical_inequality}}

Proposition \ref{Prop:critical_inequality} is proved using the functional realization of the dual space of 
$U^-_{\bL}$ introduced in \cite{MR1275728} (see also \cite{MR1934307}, \cite{MR2215613}, \cite{MR2290922}).
Here we give a detailed exposition of the realization. 
%Though all the results in this subsection is previously known, we give (a sketch of) proof for the reader's convenience.

For $i, j \in I$, write $\hat{c}_{ij} = \max\{1-c_{ij}, 1\}$.
Let $\ggg = \sum_{i \in I} m_i \ga_i \in Q^+$,
and define $\mathcal{U}_\ggg$ by the subspace of the space of rational functions in the variables 
\[ \bm{x}_\ggg =\big\{x_r^{(i)} \bigm| i\in I, 1\leq r \leq m_i\big\},
\]
consisting of the functions $g(\bm{x}_\ggg)$ of the form
\[ g(\bm{x}_\ggg)=\frac{g'(\bm{x}_\ggg)}{\prod_{i<j} \prod_{\begin{smallmatrix} 1\leq r\leq m_i
   \\ 1\leq s \leq m_j\end{smallmatrix}}  (x_r^{(i)} - x_s^{(j)})},
\] 
where $g'(\bm{x}_\ggg) \in \C\big[\,(x_r^{(i)})^{\pm 1}\bigm| i,r\,\big]$ is a Laurent polynomial, 
symmetric under the exchange of variables $x_{r}^{(i)} \leftrightarrow x_s^{(i)}$
with the same superscript, and vanishes provided $x_1^{(i)} = x_2^{(i)} = \cdots = x_{\hat{c}_{ij}}^{(i)} = x_1^{(j)}$ for $i\neq j$.
%Note that if $c_{ij}=0$, then $g_1(\bm{x})|_{x_1^{(i)}=x_1^{(j)}}=0$ and hence $g(\bm{x})$ has no pole at $x_1^{(i)} = x_1^{(j)}$,
%which implies $g(\bm{x})$ has no pole at $x_r^{(i)} = x_s^{(j)}$ for all $r,s$ by symmetry.
We will write $\bm{x}_\ggg$ simply as $\bm{x}$ when $\ggg$ is obvious from the context.
Let $\mU = \bigoplus_{\ggg\in Q^+} \mU_\ggg$.

%Next we will define a bilinear map on $U^-_{\bL} \times \mathcal{U}$.
%In the sequel, we write $x_m^{(i)}$ for $x_{m_i}^{(i)}$ to simplify the notation.
For $\ggg \in Q^+$, $i \in I$ and $k \in \Z$, define a map $R_{i,k}\colon \mathcal{U}_\ggg \to \mathcal{U}_{\ggg-\ga_i}$ by
\[ R_{i,k}\big(g(\bm{x})\big) = \Res_{x_1^{(i)}} (x_1^{(i)})^{k} g(\bm{x}),
\]
and extend it on $\mU$ linearly.
Here the residue $\Res_{x}\, g(\bm{x})$ for a variable $x$ is defined as follows: 
first regard $g(\bm{x})$ as a formal Laurent series in $x$ by expanding all $(x-x_r^{(i)})^{-1}$ in positive power of 
$x/x_r^{(i)}$, and then take the coefficient of $x^{-1}$. 
%In the analytical language, 
%\[ \Res_x g(\bm{x}) = \frac{1}{2\pi\sqrt{-1}}\int_{C_x} gdx,
%\]
%where $C_x$ is the 

\begin{Rem}\normalfont
 Precisely to say, in order to define $R_{i,k}\big(g(\bm{x})\big)$ as a function in $\mathcal{U}_{\ggg-\ga_i}$,
 we need to reindex the variables $\{x_2^{(i)},\ldots,x_{m_i}^{(i)}\}$ to $\{x_1^{(i)},\ldots,x_{m_i-1}^{(i)}\}$ 
 (any reindex produces a same function because of the symmetry). 
 In the sequel, we always assume such an obvious reindexing is done, if necessary, without any mention.
\end{Rem}

If a rational function $g(x,y)$ in two variables $x$ and $y$ has no poles except at $x=0$, $y=0$ and $x=y$, 
it follows  from the Cauthy's residue theorem that
\begin{equation}\label{eq:Cauthy's_residue_thm}
 \big(\Res_{y} \Res_{x}- \Res_{x}\Res_{y}\big)g(x,y)= -\Res_{y}\Res_{x=y}\,g(x,y).
\end{equation}
From this, a useful formula is deduced as follows.
Suppose that $h \in \Z_{>0}$, $i_1,\ldots,i_h \in I$ and $k_1,\ldots,k_h \in \Z$ are given.
Set $\gb=\ga_{i_1}+\cdots +\ga_{i_h}$,
and rename the variables $\bm{x}_{\gb}$ into $\{x_1,\ldots,x_h\}$ by 
\begin{equation*}\label{eq:rename}
 x_r = x^{(i_r)}_{\#\{s \leq r \mid\, i_s=i_r\}} \ \text{ for } 1 \leq r \leq h.
\end{equation*}
Then in view of (\ref{eq:Cauthy's_residue_thm}), we have
\[ \begin{split}
 [R_{i_h,k_h},[&R_{i_{h-1},k_{h-1}},\ldots,[R_{i_{2},k_{2}},R_{i_1,k_1}]\!\ldots]]\,g(\bm{x})\\
 &=(-1)^{h-1}\Res_{x_h}\Big( x_h^{k_h}\Res_{x_{h-1}=x_h}\Big(x_{h-1}^{k_{h-1}}\cdots \Res_{x_1=x_2} \Big(x_1^{k_1} g(\bm{x})\Big)\!\cdots
 \!\Big)\!\Big)
\end{split}\]
for $g(\bm{x}) \in \mU$.
Moreover for each $2\leq r \leq h-1$, it can be proved that the function 
\[ x_r^{k_r}\Res_{x_{r-1}=x_r} \Big(x_{r-1}^{k_{r-1}}\cdots \Res_{x_1=x_2}\Big( x_1^{k_1}g(\bm{x})\Big)\!\cdots\!\Big)
\]
%on which $\Res_{x_r=x_{r+1}}$ is applied 
has at most a simple pole at $x_r=x_{r+1}$ (we will give a proof of this fact 
in Appendix \ref{subsection:simplicity} since we have not found one in the literature).
Hence the above formula is rewritten in the following simpler form:
\begin{equation}\label{eq:convenient}
  \begin{split}
   [R_{i_h,k_h},[R_{i_{h-1},k_{h-1}}&,\ldots,[R_{i_2,k_2},R_{i_1,k_1}]\!\ldots]]\,g(\bm{x}) \\
   &=\Res_{x_h} \Big\{\prod_{r=2}^{h} (x_{r}-x_{r-1}) g(\bm{x})\Big|_{ x_1=x_2=\cdots=x_h}\cdot (x_h)^{k_1+\cdots+k_h}\Big\}.
 \end{split}
\end{equation}
%(recall that $\Res_{x=y}\, g(x,y) = (x-y)g(x,y)|_{x=y}$ holds if $g(x,y)$ has at most a simple pole at $x=y$).
%if $f_\ga =[f_{i_h},[f_{i_{h-1}},\ldots,[f_{i_2},f_{i_1}]\!\ldots]]$,
%then for $g(\bm{x}) \in \mU$ we have
%\begin{equation}\label{eq:convenient}
% \begin{split}
%  \langle  f_{\ga,k}, g(\bm{x})\rangle = \Res_{x^{(i_m)}_{r_m}} \Big\{\prod_{a=2}^{m} 
%  (x_1^{(i_1)}-x_{r_{a}}^{(i_{a})}) 
%  g(\bm{x})\Big|_{ x^{(i_1)}_{r_1}=x^{(i_2)}_{r_2} =\cdots= x^{(i_m)}_{r_m}}\cdot (x^{(i_m)}_{r_m})^{k-m}\Big\},
%  \langle  f_{\ga,k}, g(\bm{x})\rangle = \Res_{y_h} \Big\{\prod_{a=2}^{h} 
%  (y_1-y_a) 
%  g(\bm{x})\Big|_{ y_1=y_2=\cdots=y_h}\cdot (y_h)^{k-h}\Big\},
% \end{split}
%\end{equation}
%where we write $(y_1,\ldots,y_h)$ for $(x_1^{(i_1)},\ldots,x_{m_{i_h}}^{(i_h)})$.
%In particular, if the sum of the orders of the poles of $g(\bm{x})$ at $y_a=y_b$ ($1\leq a < b \leq h$) is less than 
%$h-1$, then the left-hand side of (\ref{eq:convenient}) is $0$.

%Note that $R_{i,k}$ maps each $\mU_\ggg$ to $\mU_{\ggg-\ga_i}$.
In the sequel, we write $f_{\ga,k} = f_\ga \otimes t^k$ and $f_{i,k} = f_i \otimes t^k$ to shorten the notation.
Define a bilinear map $\langle \ , \ \rangle \colon U^-_{\bL} \times \mU \to \mU$ by 
\[ \Big\langle f_{i_1,k_1}f_{i_2,k_2}\cdots f_{i_N,k_N}, g(\bm{x})\Big\rangle = R_{i_1,k_1}R_{i_2,k_2}\cdots R_{i_N,k_N}\,g(\bm{x}),
\]
where $i_1,\ldots,i_N \in I$ and $k_1,\ldots,k_N \in \Z$.
%The following proposition is well-known.

\begin{Prop}\label{Prop:bilinear_well}
 The bilinear map $\langle \ , \ \rangle$ is well-defined. 
\end{Prop}

\begin{proof}
 Since $U^-_{\bL}$ is generated by $\{f_{i,k}\mid i \in I, k \in \Z\}$ with relations %($h \in\Z_{\geq 2}$, 
% $i,j,i_a \in I$, $l,k_a,\ell_a \in \Z$)
 \begin{align*}
  [f_{i,k_{\hat{c}_{ij}}},[f_{i,k_{\hat{c}_{ij}-1}},\ldots,[f_{i,k_1},f_{j,l}]\!\ldots]] &= 0 \ \ (i,j \in I, k_r,l \in \Z), \ \text{and}\\
  [f_{i_h,k_h},[f_{i_{h-1},k_{h-1}},\ldots,[f_{i_2,k_2},f_{i_1,k_1}]\!\ldots]] &= [f_{i_h,l_h},[f_{i_{h-1},l_{h-1}},\ldots,
  [f_{i_2,l_2},f_{i_1,l_1}]\!\ldots]]\\
  (h \in \Z_{>0}, i_r \in I, k_r, l_r& \in \Z \text{ such that } k_1+\cdots+k_h = l_1+\cdots + l_h),  
 \end{align*}
 it suffices to show for arbitrary $g \in \mathcal{U}$ that
 \begin{align*}
  [R_{i,k_{\hat{c}_{ij}}},[R_{i,k_{\hat{c}_{ij}-1}},\ldots,[R_{i,k_1}, R_{j,l}]\!\ldots]]\,g &= 0, \ \text{and}\\
  [R_{i_h,k_h},[R_{i_{h-1},k_{h-1}},\ldots,[R_{i_2,k_2},R_{i_1,{k_1}}]\!\ldots]]\,g\!
   &=\! [R_{i_h,l_h},[R_{i_{h-1},l_{h-1}},\ldots,[R_{i_2,l_2},R_{i_1,l_1}]\!\ldots]]\, g.
 \end{align*}
 These are easily deduced from (\ref{eq:convenient}) and the definition of $\mU$, and hence the assertion is proved.
\end{proof}

For $\ggg\in Q^+$, $\langle \ , \ \rangle$ defines a pairing $\big(U^-_{\bL}\big)_{-\ggg} \times \mU_\ggg \to \C$.
Note that $\big(U^-_{\bL}\big)_{-\ggg}$ is naturally $\Z$-graded
by the degree of polynomials, and so is $\mU_\ggg$ by the total degree.
%Define a $\Z$-grading on $\mathcal{U}_\ggg$ by setting $\deg x_r^{(i)} =1$ and $(x_r^{(i)}-x_s^{(j)})^{-1} = -1$.
It is easily checked that, 
%the pairing $\langle \ ,\ \rangle$ is graded, that is, 
for $F \in \big(U^-_{\bL}\big)_{-\ggg}^k$ and $g \in 
\mathcal{U}_\ggg^l$, $\langle F,g\rangle = 0$ unless $k + l=-\het(\ggg)$.

Let $\ol{\mathcal{U}}_\ggg$ be the subspace of $\mathcal{U}_\ggg$ consisting of $g(\bm{x})$ such that 
\[ \langle f_{i,k}\big(U^-_{\bL}\big)_{-\ggg+\ga_i}, g(\bm{x}) \rangle=0 \ \ \ \text{for all} \ i\in I, k \in \Z_{\le 0}.
\]
By the definition of the pairing, this is equivalent to that, if we regard $g(\bm{x})$ as a formal Laurent series in $(x_1^{(i)})^{-1}$
by expanding all $(x_1^{(i)}-x_s^{(j)})^{-1}$ in positive power of $x_s^{(j)}/x_1^{(i)}$, then the coefficient of 
$(x_1^{(i)})^{-k-1}$ is $0$ for all $k\in \Z_{\leq 0}$.
Hence the following lemma holds.

\begin{Lem}\label{Lem:description_of_barU}
  We have
  \[ \ol{\mathcal{U}}_\ggg = \big\{ g(\bm{x}) \in \mathcal{U}_\ggg \bigm| \deg_{x_1^{(i)}}  g(\bm{x}) \leq -2 \text{ for all } i \in I \big\},
  \]
 where $\deg_{x_1^{(i)}}$ is defined by setting $\deg_{x_1^{(i)}} x_r^{(j)} = \gd_{(i,1),(j,r)}$ and 
 \[ \deg_{x_1^{(i)}} (x_r^{(j)} - x_s^{(j')})^{-1} = \begin{cases} -1 & \text{ if } (j,r) = (i,1) \text{ or } (j',s)=(i,1),\\
                                                                  0  & \text{ otherwise}.
                                 \end{cases}
 \]
\end{Lem}

Set $\ol{U}^-_{\bL} = U^-_{\bL}/(\fn_-[t^{-1}]U^-_{\bL})$.
Obviously $\langle \ , \ \rangle$ induces a pairing between $(\ol{U}^-_{\bL})_{-\ggg}$ and $\ol{\mathcal{U}}_\ggg$,
and moreover the following proposition holds. 
%*Since we have not found a proof for general $\fg$ in the literature, we give one here

\begin{Prop}\label{Prop:dual_isom}
 The pairing induces a linear isomorphism $\ol{\mathcal{U}}_\ggg \stackrel{\sim}{\to} (\ol{U}^-_{\bL})_{-\ggg}^\vee$,
 where $(\ol{U}^-_{\bL})_{-\ggg}^\vee$ denotes the restricted dual $\bigoplus_{k \in \Z} \big((\ol{U}^-_{\bL})_{-\ggg}^k\big)^*$.
\end{Prop}

This proposition is proved in Appendix \ref{subsection:pf_of_dual}.

\subsection{Proof of Proposition \ref{Prop:critical_inequality}}\label{subsection:current_complete}

Now we return to the setting of Subsection \ref{Subsection:reduction}.
%Recall that $i_1,\ldots,i_p \in I$ and $\ell_1,\ldots,\ell_p$ are fixed.
Through the isomorphism $\ol{\mathcal{U}}_\ggg\stackrel{\sim}{\to} (\ol{U}_{\bL}^-)_{-\ggg}^\vee$, 
the dual space of $\Big(U^-_{\bL}\big/\big(\fn_-[t^{-1}]U^-_{\bL} + \mathcal{J} \big)\Big)_{-\ggg}$ is regarded as a subspace of 
$\ol{\mathcal{U}}_\ggg$, consisting of the functions $g(\bm{x})$ satisfying $\langle \mathcal{J}, g(\bm{x})\rangle = 0$.
Set $\hat{F}_i(z) = \sum_{k=-\infty}^\infty f_{i,k}z^{-k-1}$. 
Then for $r>0$, we have
\[ \langle \hat{F}_i(z)^r, g(\bm{x})\rangle = g(\bm{x})\big|_{x_1^{(i)} =  \cdots = x_r^{(i)}=z}.
\]
(Here the coefficients of $\hat{F}_i(z)^r$ belong to a completion of $U^-_{\bL}$, but the pairing is still well-defined.)
Let $\mathcal{V}_\ggg$ denote the subspace of $\ol{\mathcal{U}}_\ggg$ consisting of the functions 
$g(\bm{x})$ such that, for every $i \in I$ and $r>0$, the order of the pole of $g(\bm{x})\big|_{x_1^{(i)} =  \cdots 
= x_r^{(i)}=z}$ at $z=0$ is at most $\sum_{k \in S_i}\min\{r,\ell_k\}$.
Then it follows from the definition of $\mJ$ that
\begin{align}\label{eq:isom_of_V} 
 \Big(U^-_{\bL}\big/\big(\fn_-[t^{-1}]U^-_{\bL} + \mathcal{J} \big)\Big)_{-\ggg}^*\cong \mV_\ggg.
\end{align}
Moreover, an upper bound of the dimension of $\mV_\ggg$ is given as follows.

\begin{Lem}\label{Lem:fermionic_form}
 Let $\gl = \sum_k \ell_k \varpi_{i_k}$. For $\ggg \in Q^+$ such that $\gl - \ggg \in P^+$, we have
 \begin{equation}\label{eq:fermionic_form}
  \dim \mV_\ggg \leq \sum_{\{m_a^{(i)}\}}\prod_{\begin{smallmatrix}i \in I\\ a>0\end{smallmatrix}} \begin{pmatrix} 
  p_a^{(i)} + m_a^{(i)} \\ m_a^{(i)}
  \end{pmatrix},
 \end{equation}
 where
 \[ p_a^{(i)} = \sum_{k \in S_i} \min\{a,\ell_k\} + \sum_{\begin{smallmatrix}j \in I\setminus\{i\}\\ b >0\end{smallmatrix}} 
    \min\{|c_{ji}|a,|c_{ij}|b\}m_b^{(j)} -2\sum_{b>0} \min\{a,b\}m_b^{(i)},
 \]
 and the sum $\sum_{\{m_a^{(i)}\}}$ is taken over $\{m_a^{(i)}\in \Z_{\geq 0}\mid i\in I, a>0\}$ satisfying $p_a^{(i)} \geq 0$ for all $i,a$, 
 and $\ggg= \sum_{i,a} am_a^{(i)}\varpi_i$.
\end{Lem}

This lemma is proved by a similar argument given in \cite{MR2290922}.
For the reader's convenience, we reproduce it in Appendix \ref{subsection:fermionic_form}.

By combining several results in \cite{MR1745263,MR1993360,MR2254805,MR2428305} (see \cite[Subsections 2.3 and 2.4]{MR2428305} for the
details), it is shown that the right-hand side of 
(\ref{eq:fermionic_form}) is equal to the $U_q(\fg)$-module multiplicity 
\[ \Big[W^{i_1,\ell_1}_q(a_1) \otimes \cdots \otimes W^{i_p,\ell_p}_q(a_p): V_q(\gl-\ggg)\Big].
\]
%where $V_q(\mu)$ denotes the simple $U_q(\fg)$-module with highest weight $\mu$.
Since this multiplicity coincides with the $\fg$-module multiplicity $\Big[W: V(\gl-\ggg)\Big]$,
Proposition \ref{Prop:critical_inequality} now follows from Lemma \ref{Lem:fermionic_form} and (\ref{eq:isom_of_V}).
This completes the proof of Theorem \ref{Thm:current}, as explained above.

\section{Proof of Proposition \ref{Prop:Uqsl2}}\label{section:proof_of_Uqsl2}

%The purpose of this section is to prove Proposition \ref{Prop:Uqsl2}, and so let us assume $\fg = \mathfrak{sl}_2$ throughout the section.
%For simplicity, we further assume that the symmetrizing matrix $\hat{D}$ is $\begin{pmatrix} 1 & 0 \\ 0 & 1\end{pmatrix}$,
%and prove the proposition in this case only.
%For a general $\hat{D}$, the same proof goes through with obvious modifications which we left to the reader.

\subsection{Quantum loop algebra of type \boldmath{$A_1^{(1)}$}}\label{Subsection:quantum_sl2}

Throughout this section we focus on the case $\fg = \sl_2$ only.
In what follows, we fix a positive integer $d$, and denote by $U_q(\bL\sl_2)$ the quantum loop algebra
associated with $\hat{C} = \begin{pmatrix} 2 & -2 \\ -2 & 2 \end{pmatrix}$ and $\hat{D} = \begin{pmatrix}
d & 0 \\ 0 & d\end{pmatrix}$ (which implies $q_0=q_1=q^d$). 
Let $\tilde{q} = q^d$.

%We write $W^{\ell}_q(a)$ for the $U_q(\bL\sl_2)$-module $W^{1,\ell}_q(a)$.
Here we collect several results concerning finite-dimensional modules over $U_q(\bL\mathfrak{sl}_2)$
(recall that we write $W^{\ell}_q(a)$ for $W^{1,\ell}_q(a)$).

\begin{Lem}\label{Lem:Uqsl2}
 {\normalfont (i)} $W^{\ell}_q(a)$ is simple as a $U_q(\sl_2)$-module and of dimension $\ell+1$.\\
 {\normalfont (ii)} The $U_q(\bL\sl_2)$-submodule of $W^{\ell_1}_q(a_1)\otimes W^{\ell_2}_q(a_2)$ generated by the tensor product of 
  $\ell$-highest weight vectors is proper and simple if
  \[ a_2/a_1 \in \{\tilde{q}^{2k}\mid \max\{\ell_1-\ell_2,0\}< k \leq \ell_1\}.
  \]
 {\normalfont (iii)} $W^{\ell_1}_q(a_1) \otimes W^{\ell_2}_q(a_2)$ is simple if 
  \[ a_2/a_1 \notin \{\tilde{q}^{2k} \mid -\ell_2 \leq k< \min\{\ell_1-\ell_2,0\}, \,\max\{\ell_1-\ell_2,0\}<k \leq \ell_1\}.
  \]
  In particular if this condition holds, then it follows from Lemma \ref{Lem:commutativity_of_tensor} (i) that 
  \[ W^{\ell_1}_q(a_1) \otimes W^{\ell_2}_q(a_2) \cong W^{\ell_2}_q(a_2) \otimes W^{\ell_1}_q(a_1).
  \]
 {\normalfont(iv)} %Let $a \in \C(q)$, $\ell \in \Z_{>0}$, and $\ell_1,\ldots,\ell_p \in \Z_{>0}$ be such that 
  If $\ell_1+\cdots +\ell_p = \ell$, then the $U_q(\bL\sl_2)$-submodule of 
  \[ W^{\ell_1}_q(a) \otimes W^{\ell_2}_q(\tilde{q}^{2\ell_1}a) \otimes \cdots \otimes W^{\ell_p}_q\big(\tilde{q}^{2(\ell_1+\cdots
     + \ell_{p-1})}a\big)
  \]
  generated by the tensor product of $\ell$-highest weight vectors is isomorphic to $W^{\ell}_q(a)$.
\end{Lem}
\begin{proof}
 The proofs of (i)--(iii) are found in \cite{MR1137064} (note that the coproduct in the paper is different from ours).
 Then (iv) follows from (ii).
\end{proof}

The modules $W^1_q(a)$ are called \textit{fundamental modules}.
The following lemma is deduced from (the proof of) \cite[Lemma 4.10]{MR1137064} (see also \cite[Theorem 2.6]{MR1883181}).

\begin{Lem}\label{Lem:l_h.w}
 The tensor product of fundamental modules $W_q^1(a_1) \otimes \cdots \otimes W_q^1(a_p)$ is $\ell$-highest weight if 
 \[ a_s/a_r \neq \tilde{q}^2 \ \text{ for all } \ 1\leq r < s \leq p.
 \]

\end{Lem}

\subsection{Realization of fusion products}\label{Subsection:Demazure}

Write $e,f,h$ for $e_1,f_1,h_1 \in \sl_2$, and $W^{\ell}$ for the $\sl_2[t]$-module $W^{1,\ell}$.
Here we will recall a certain realization of a fusion product of $W^{\ell}$'s, and for that we prepare some notation.

%Write $e,f,h$ for $e_1,f_1,h_1 \in \sl_2$.
%, and identify the weight lattice of $\sl_2$ with $\Z$ as usual.
Let $\wh{\sl_2} = \sl_2 \otimes \C[t,t^{-1}] \oplus \C K$ be the affine Lie algebra of type $A_1^{(1)}$ 
(without a degree operator).
Here $K$ is the canonical central element.
% and $d$ is the degree operator.
Note that $\sl_2$ and the current algebra $\sl_2[t]$ are naturally regarded as Lie subalgebras of $\wh{\sl_2}$.
Let $\hfh = \C h \oplus \C K$,
and $\gL_0,\gL_1 \in \hfh^*$ be the fundamental weights defined by 
\[ \langle h,\gL_0\rangle =0, \ \ \langle h,\gL_1\rangle = 1,\ \ \langle K,\gL_0\rangle = \langle K,\gL_1 \rangle= 1.
\]
Define a Borel subalgebra $\hfb \subseteq \wh{\sl_2}$ by
\[ \hfb = \hfh \oplus \C e \oplus t\sl_2[t].
\]
Let  $\hfp_1$ and $\hfp_0$ be the Lie subalgebras of $\wh{\sl_2}$ defined respectively by
\[ \hfp_1 = \C f \oplus \hfb =\sl_2[t] \oplus \C K, \ \ \ \hfp_0 = \C (e\otimes t^{-1}) \oplus \hfb,
\]
which are minimal parabolic subalgebras.
Let $\tau$ be the Lie algebra automorphism on $\wh{\sl_2}$ induced from the unique nontrivial Dynkin diagram automorphism.
Explicitly, $\tau$ is defined as follows:
\[ \tau(e \otimes t^k) = f \otimes t^{k+1}, \ \ \tau(h \otimes t^k) = -h \otimes t^k + \gd_{k0} K, \ \ \tau(K)=K, \ \ \tau^2 = \id.
\]

Given a finite-dimensional $\hfp_1$-module $D$ which is $\hfh$-semisimple,
we define a new $\hfp_1$-module $F(D)$ as follows.
Let $\tau^* D$ be the pull-back with respect to $\tau$, which is a $\hfp_0$-module since $\tau(\hfp_0) = \hfp_1$.
We consider $\tau^*D$ as a $\hfb$-module by restriction, and then 
$F(D)$ is defined by the unique maximal finite-dimensional $\hfp_1$-module quotient of the induced module 
$U(\hfp_1) \otimes_{U(\hfb)} \tau^*(D)$ (which exists by \cite[Proposition 2.1]{MR1729359}).

For $\ell \in \C$, denote by $\C_{\ell \gL_0}$ the $1$-dimensional $\hfp_1$-module on which $K$ acts as a scalar multiplication by $\ell$
and $\sl_2[t]$ acts trivially.
Now the following lemma is a reformulation of \cite[Theorem 2.5]{MR1729359}
(for the present formulation, see \cite[Theorem 6.1]{MR2964614}).

%be the Lie algebra homomorphism defined by
%\[ \Phi(e\otimes t^k) = e\otimes t^{k+1}, \ \ \Phi(h\otimes t^k) = h \otimes t^k, \ \ \Phi(f\otimes t^k) = f\otimes t^{k-1}.
%\]
%Given a finite-dimensional $\sl_2[t]$-module $D$ on which $h$ acts diagonalizablly and $\ell \in \Z_{\geq 0}$,
%we define a $\sl_2[t]$-module $\mathcal{T}_\ell(D)$ as follows:
%let $\Phi^* D$ be the pull-back with respect to $\Phi$, which is a $\fu$-module.
%Then $\mathcal{T}_\ell(D)$ is defined by the unique maximal finite-dimensional $\sl_2[t]$-module quotient of the induced module 
%\[ U(\sl_2[t]) \otimes_{U(\fu)} \big( \C_\ell \otimes \Phi^* D\big).
%\]

\begin{Prop}\label{Prop:Feigin_Loktev}
 Given a partition $(\ell_1\geq \cdots \geq \ell_{p-1} \geq \ell_p)$, it follows that
 \[ W^{\ell_1} * \cdots * W^{\ell_{p-1}}* W^{\ell_p} \cong F\Big(\C_{(\ell_1-\ell_2)\gL_0} \otimes 
    \cdots \otimes F\big(\C_{(\ell_{p-1}-\ell_p)\gL_0} \otimes F(\C_{\ell_p\gL_0})\big)\cdots\Big)
 \]
 as $\sl_2[t]$-modules.
\end{Prop}

We need a slightly alternative realization. 
Let $\hat{V}_0$ (resp.\ $\hat{V}_1$) be the simple highest weight $\wh{\sl_2}$-module with highest weight $\gL_0$ (resp.\ $\gL_1$).
Let $m \in \Z_{\ge 0}$. If $m$ is even (resp.\ odd), let $v_m$ denote an extremal weight vector of $\hat{V}_0$ (resp.\ $\hat{V}_1$)
with weight $m \gL_1 - (m-1)\gL_0$.
% (here $\varpi_1$ is regarded as an element of $\hfh^*$ by setting $\langle K, \varpi_1\rangle=0$).
Note that $\tau^*(V_0) \cong V_1$ and $\tau^*(V_1) \cong V_0$ hold,
and these isomorphisms map $\tau^*(v_m)$ to an extremal weight vector with weight $m \gL_0 - (m-1)\gL_1$,
which we denote by $v_m^{-}$.
It is easily checked that 
\begin{equation}\label{eq:extremal}
  (f\otimes t)v_m^{-} = 0, \ \ (e\otimes t^{-1})^{m}v_m^{-} \in \C^\times v_{m+1}, \ \ (e\otimes t^{-1})^{m+1}v_m^{-} = 0.
\end{equation}

For a sequence $m_1,\ldots,m_p$ of positive integers, define an $\sl_2[t]$-module $D(m_1,\ldots,m_p)$ by
\[ D(m_1,\ldots,m_p) = U(\sl_2[t])(v_{m_1} \otimes \cdots \otimes v_{m_p}) \subseteq \hat{V}_{\bar{m}_1} \otimes \cdots \otimes 
   \hat{V}_{\bar{m}_p},
\]
where $\bar{m}$ is $0$ if $m$ is even, and $1$ if $m$ is odd.

\begin{Lem}\label{Lem:another_realization}
 For a sequence $\bm{\ell} = (\ell_1,\ldots,\ell_p)$ of positive integers, set
 \[  L =\max\{\ell_1,\ldots,\ell_p\} \ \text{ and } \
     m_j = \#\,\{ 1\leq k \leq p \mid \ell_k \geq j\} \, \text{ for } \, 1 \leq j \leq L.
 \]
 Then we have 
 \[ W^{\ell_1} * \cdots * W^{\ell_p} \cong D(m_1,\ldots,m_L)
 \]
 as $\sl_2[t]$-modules.
\end{Lem}

\begin{proof}
 Without loss of generality we may assume that $\bm{\ell}$ is a partition,
 and then by Proposition \ref{Prop:Feigin_Loktev} it suffices to show that
 \begin{equation}\label{eq:isomorphism}
  F\Big(\C_{(\ell_1-\ell_2)\gL_0} \otimes 
  \cdots \otimes F\big(\C_{(\ell_{p-1}-\ell_p)\gL_0} \otimes F(\C_{\ell_p\gL_0})\big)\cdots\Big) \cong D({}^t\bm{\ell}),
 \end{equation}
 where ${}^t\bm{\ell}$ is the transposed partition.
 We will show this by the induction on $p$. The case $p=1$ is easily checked.
 Assume $p >1$, and set $\bm{\ell}' = (\ell_2, \ldots,\ell_p)$.
% Write $T_j = T_j(\bm{\ell})$ and $T_j' = T_j(\bm{\ell}')$. 
% Note that $|T_j'| = |T_j| -1$ holds for $1\leq j \leq \ell_1$.
 Since 
 \[ D({}^t\bm{\ell}') = U(\hfp_1)(v_{m_1-1} \otimes \cdots \otimes v_{m_{\ell_2}-1}),
 \]
 we have
 \begin{align*}
  \tau^*\big(\C_{(\ell_1 -\ell_2)\gL_0}\otimes D({}^t\bm{\ell}')\big) &\cong 
  \tau^*\Big(U(\hfp_1)(v_0^{\otimes (\ell_1 - \ell_2)} \otimes v_{m_1-1} \otimes \cdots \otimes v_{m_{\ell_2}-1})\Big)\\
  &\cong U(\hfp_0)(v_1^{\otimes (\ell_1-\ell_2)} \otimes 
  v_{m_1-1}^- \otimes \cdots \otimes v_{m_{\ell_2}-1}^-)\\
  &= U(\hfb)\big(v_1^{\otimes(\ell_1-\ell_2)} \otimes v_{m_1} \otimes \cdots \otimes v_{m_{\ell_2}}\big),
 \end{align*}
 where  the equality follows from (\ref{eq:extremal}).
 Hence by the definition of $F$, there exists a surjective $\hfp_1$-module homomorphism
 \[ F\big(\C_{(\ell_1 -\ell_2)\gL_0}\otimes D({}^t\bm{\ell}')\big) \twoheadrightarrow 
    U(\hfp_1)\big(v_1^{\otimes(\ell_1-\ell_2)} \otimes v_{m_1} \otimes \cdots \otimes v_{m_{\ell_2}}\big)\cong D({}^t\bm{\ell}),
 \]
 which induces a surjection 
 \[  F\Big(\C_{(\ell_1-\ell_2)\gL_0} \otimes 
     \cdots \otimes F\big(\C_{(\ell_{p-1}-\ell_p)\gL_0} \otimes F(\C_{\ell_p\gL_0})\big)\cdots\Big) \twoheadrightarrow D({}^t\bm{\ell})
 \]
 by the induction hypothesis.
 Since the dimensions of these modules coincide by \cite[Theorem 5]{MR1887117} and \cite[Corollary 6.2]{MR2964614},
 this is an isomorphism.
 Hence (\ref{eq:isomorphism}) is proved, as required.
\end{proof}

Finally we recall the following proposition.

\begin{Prop}\label{Prop:Chari_Loktev}
 Assume that $a_1,\ldots,a_p \in \mA^\times$ satisfy $a_1(1) = \cdots= a_p(1) = c \in \C^\times$, and 
 $W_q^1(a_1) \otimes \cdots \otimes W_q^1(a_p)$ is $\ell$-highest weight.
 Then we have
 \[ \ol{W_q^1(a_1) \otimes \cdots \otimes W_q^1(a_p)} \cong \varphi_c^* D(p)
 \]
 as $\sl_2[t]$-modules.
\end{Prop}

\begin{proof}
 This follows from \cite[Theorem 5]{MR1850556} and \cite[Corollary 1.5.1]{MR2271991}.
\end{proof}

\subsection{Proof of Proposition \ref{Prop:Uqsl2}}\label{subsection:proof_of_Uqsl2}

As in Proposition \ref{Prop:Uqsl2}, let $\ell_1,\ldots,\ell_p \in \Z_{>0}$ and $a_1,\ldots,a_p \in \mA^\times$ be such that 
$a_1(1) = \cdots = a_p(1) = c \in \C^\times$
and $W^{\ell_1}_q(a_1)\otimes \cdots \otimes W_q^{\ell_p}(a_p)$ is $\ell$-highest weight.
%For $a,b \in \C(q)$, we write $a< b$ if and only if $b/a \in \tilde{q}^{2\Z_{> 0}}$.

\begin{Lem}
 There exists a permutation $\gs$ on the set $\{1,\ldots,p\}$ satisfying
 \[ a_{\gs(s)}/a_{\gs(r)} \notin \tilde{q}^{2\Z_{> 0}} \ \text{ for all } \ 1\leq r < s \leq p,
 \]
 and
 \[ W^{\ell_1}_q(a_1) \otimes \cdots \otimes W^{\ell_p}_q(a_p) \cong W^{\ell_{\gs(1)}}_q(a_{\gs(1)}) \otimes \cdots \otimes W_q^{\ell_{\gs(p)}}
    (a_{\gs(p)}).
 \]
\end{Lem}

\begin{proof}
 We show the assertion by the induction on $p$. There is nothing to prove when $p =1$.
 Assume that $p>1$.
 By the induction hypothesis, we may assume that 
 \[ a_{s}/a_{r} \notin \tilde{q}^{2\Z_{> 0}} \ \text{ for } \ 1\leq r < s \leq p-1.
 \]
 If $a_p/a_{p-1} \in \tilde{q}^{2\Z_{\leq 0}}$, then the assertion holds with $\gs = \id$, since $a_p/a_{p-1} \in \tilde{q}^{2\Z_{\leq 0}}$
 and $a_{p-1}/ a_r \notin \tilde{q}^{2\Z_{>0}}$ imply $a_p/a_r \notin \tilde{q}^{2\Z_{>0}}$.
 Assume that $a_p/a_{p-1} \notin \tilde{q}^{2\Z_{\leq 0}}$.
 If 
 \[ a_p/a_{p-1} \in \big\{\tilde{q}^{2k} \bigm| \max\{\ell_{p-1} -\ell_p,0\} < k \leq \ell_p\big\},
 \]
 then the submodule of $W^{\ell_{p-1}}_q(a_{p-1})
 \otimes W^{\ell_p}_q(a_p)$ generated by the tensor product of $\ell$-highest weight vectors is proper by Lemma \ref{Lem:Uqsl2} (ii),
 which contradicts that $W^{\ell_1}_q(a_1) \otimes \cdots \otimes W^{\ell_p}_q(a_p)$ is $\ell$-highest weight.
 Hence $W^{\ell_{p-1}}_q(a_{p-1}) \otimes W^{\ell_p}_q(a_p)$ is simple by Lemma \ref{Lem:Uqsl2} (iii), and we have
 \[ W^{\ell_1}_q(a_1) \otimes \cdots \otimes W^{\ell_{p-1}}_q(a_{p-1}) \otimes W^{\ell_p}_q(a_p) \cong W^{\ell_1}_q(a_1) \otimes \cdots
    \otimes W^{\ell_p}_q(a_p) \otimes W^{\ell_{p-1}}_q(a_{p-1}).
 \]
 Now by applying the induction hypothesis to $W_q^{\ell_1}(a_1) \otimes \cdots \otimes W_q^{\ell_{p-2}}(a_{p-2}) \otimes W_q^{\ell_p}(a_p)$,
 we obtain the required result.
\end{proof}

By this lemma, we may (and do) assume that 
\begin{equation}\label{eq:assumption}
 a_s/a_r \notin \tilde{q}^{2\Z_{>0}} \ \text{ for all } \ 1 \leq r<s \leq p.
\end{equation}
Put $L = \max\{\ell_k \mid 1\leq k \leq p\}$, and
\[ M_j = \{1\leq k \leq p \mid \ell_k \geq j\} \ \text{ for } \ 1\leq j \leq L.
\]

\begin{Lem}\label{Lem:injection}
 There exists an injective $U_q(\bL\sl_2)$-module homomorphism 
 \[ \iota\colon W^{\ell_1}_q(a_1)\otimes \cdots \otimes W^{\ell_p}_q(a_p) 
    \hookrightarrow \bigotimes_{k \in M_1}W_q^1(a_k) \otimes \bigotimes_{k \in M_2}
    W_q^1(\tilde{q}^2a_k) \otimes \cdots \otimes \bigotimes_{k \in M_{L}} W_q^1(\tilde{q}^{2L-2}a_k)
 \]
 mapping an $\ell$-highest weight vector to a tensor product of $\ell$-highest weight vectors.
 Here each $\bigotimes_k W_q^1(\tilde{q}^{2j}a_k)$ are ordered so that $W_q^1(\tilde{q}^{2j}a_r)$ is left to $W_q^1(\tilde{q}^{2j}a_s)$
 if $r<s$.
\end{Lem}

\begin{proof}
 We show the assertion by the induction on $p$. If $p=1$, it follows from Lemma \ref{Lem:Uqsl2} (iv).
 Assume that $p>1$.
 We claim that, for each $1\leq j \leq \ell_1-1$, there exists an injective homomorphism
 \begin{equation}\label{eq:injections}
   W_q^{\ell_1+1-j}(\tilde{q}^{2j-2}a_1) \otimes \bigotimes_{k \in M_j\setminus\{1\}} W_q^1(\tilde{q}^{2j-2}a_k)
   \hookrightarrow \bigotimes_{k \in M_j} W_q^1(\tilde{q}^{2j-2}a_k) \otimes W_q^{\ell_1-j} (\tilde{q}^{2j}a_1).
 \end{equation}
 Indeed, putting $T= \{ k \in M_j \mid a_k = a_1\}$, it follows from (\ref{eq:assumption}) and Lemma \ref{Lem:Uqsl2} that
 \begin{align*} 
  W_q^{\ell_1+1-j}(&\tilde{q}^{2j-2}a_1) \otimes \bigotimes_{k \in M_j\setminus\{1\}} W_q^1(\tilde{q}^{2j-2}a_k)\\ 
  &\stackrel{\sim}{\to}
  W_q^1(\tilde{q}^{2j-2}a_1)^{\otimes(\# T-1)} \otimes W_q^{\ell_1+1-j}(\tilde{q}^{2j-2}a_1) \otimes \bigotimes_{k \in M_j\setminus T}
  W_q^1(\tilde{q}^{2j-2}a_k)\\
  &\hookrightarrow W_q^1(\tilde{q}^{2j-2}a_1)^{\otimes \# T} \otimes W_q^{\ell_1-j}(\tilde{q}^{2j}a_1) \otimes 
  \bigotimes_{k\in M_j\setminus T} W_q^1(\tilde{q}^{2j-2}a_k)\\
  &\stackrel{\sim}{\to} \bigotimes_{k \in M_j} W_q^1(\tilde{q}^{2j-2}a_k) \otimes W_q^{\ell_1-j} (\tilde{q}^{2j}a_1),
 \end{align*}
 and hence the claim is proved.

 By composing the homomorphisms induced from (\ref{eq:injections}), we obtain an injective homomorphism
 \begin{align*}
  W_q^{\ell_1}(a_1) \otimes \bigotimes_{k \in M_1\setminus\{1\}}W_q^1(a_k) \otimes &\cdots \otimes \bigotimes_{k \in M_{L}
  \setminus\{1\}} W_q^1(\tilde{q}^{2L-2}a_k) \\ &\hookrightarrow
  \bigotimes_{k \in M_1}W_q^1(a_k) \otimes \cdots \otimes \bigotimes_{k \in M_{L}} W_q^1(\tilde{q}^{2L-2}
  a_k).
 \end{align*}
 Since there is an injective homomorphism from $W^{\ell_1}_q(a_1)\otimes \cdots \otimes W^{\ell_p}_q(a_p)$ to the
 left-hand side by the induction hypothesis, the assertion is proved.
\end{proof}

Note that, by Lemma \ref{Lem:l_h.w} and (\ref{eq:assumption}), $\bigotimes_{k\in M_j} W_q^1(\tilde{q}^{2j-2}a_k)$ are $\ell$-highest weight
for all $j$.
Let $v \in W^{\ell_1}_q(a_1)\otimes \cdots \otimes W^{\ell_p}_q(a_p)$ and $v_j \in \bigotimes_{k \in M_j}W_q(\tilde{q}^{2j-2}a_k)$ 
be $\ell$-highest weight vectors satisfying $\iota(v) = v_1\otimes \cdots \otimes v_{\ell_{L}}$.
By Lemma \ref{Lem:A-form} we have
\[ \iota(U_{\mA}(\bL\sl_2)v) \subseteq U_{\mA}(\bL\sl_2)v_1 \otimes \cdots \otimes U_{\mA}(\bL\sl_2)v_{\ell_{L}},
\]
and hence $\iota$ induces an $\bL\sl_2$-module homomorphism 
\[ \ol{\iota}\colon \ol{W_q^{\ell_1}(a_1)\otimes \cdots \otimes W_q^{\ell_p}(a_p)} \to \ol{\bigotimes_{k \in M_1}W_q^1(a_k)} \otimes 
   \cdots \otimes \ol{\bigotimes_{k \in M_{\ell_{L}}}W_q^1(\tilde{q}^{2\ell_{L}-2}a_k)}.
\]
Set $m_j = \#\,M_j$ ($1\le j \le L$). 
By Proposition \ref{Prop:Chari_Loktev}, the right-hand side is isomorphic to
\[ \varphi_c^*D(m_1) \otimes \cdots \otimes \varphi_c^* D(m_L) =
   \varphi_c^*\big(D(m_1) \otimes \cdots \otimes D(m_L)\big),
\]
and the image of the composition of this isomorphism with $\ol{\iota}$ is $\varphi_c^*D(m_1, \ldots, m_L)$,
which is isomorphic to $\varphi_c^*(W^{\ell_1} * \cdots * W^{\ell_p})$ by Lemma \ref{Lem:another_realization}.
Hence we obtain a surjective homomorphism 
\[ \ol{W_q^{\ell_1}(a_1)\otimes \cdots \otimes W_q^{\ell_p}(a_p)} \twoheadrightarrow \varphi_c^*(W^{\ell_1} * \cdots * W^{\ell_p}),
\]
and it is easy to see that the dimensions of the two modules are equal. 
Hence Proposition \ref{Prop:Uqsl2} is proved.

\appendix
\section{}\label{Appendix}

In this appendix, we will give proofs of the results mentioned in Section \ref{Section:Proof1}.
We use freely the notations introduced in the section.
For $\ggg = \sum_i m_i \ga_i \in Q^+$ we write
\[ \gD_\ggg = \prod_{i<j} \prod_{\begin{smallmatrix} 1\leq r\leq m_i\\ 1\leq s \leq m_j
   \end{smallmatrix}}(x_r^{(i)}-x_s^{(j)})
\]
for ease of notation.

\subsection{Filtration on {\boldmath$ \mU_\ggg$}}

Let $\ggg = \sum_i m_i \ga_i \in Q^+$.
Following \cite{MR2290922}, we will define a filtration on $\mU_\ggg$, which plays an important roll in the next subsections.

Let $\bm{\mu} = (\mu^{(i)})_{i \in I}$ be an $I$-tuple of partitions satisfying $|\mu^{(i)}| = m_i$,
%where each $\mu^{(i)}$ is a partition of $m_i$. 
and denote by $m_a^{(i)}$ the number of rows of length $a$ in $\mu^{(i)}$
(here, as usual, we identify partitions with Young diagrams).
%Set $\ell(\bm{\mu}) = \big(\ell(\mu^{(1)}),\ldots,\ell(\mu^{(n)})\big) \in \Z_{\ge 0}^{I}$, where $\ell(\mu)$ denotes the length of $\mu$.
Let $\bm{y}_{\bm{\mu}}$ be the set of variables indexed by the rows of $\mu^{(i)}$'s:
\[ \bm{y}_{\bm{\mu}}=\{ y_{a,u}^{(i)} \mid i \in I, a>0, 1\leq u \leq m_a^{(i)}\}.
\]
We will define a specialization map $\varphi_{\bm{\mu}}$ from $\mU_\ggg$ to the space of rational functions $\C(\bm{y}_{\bm{\mu}})$ 
in $\bm{y}_{\bm{\mu}}$. 
%variables
%\[ \bm{y}_{\bm{\mu}}=\{ y_{a,u}^{(i)} \mid i \in I, a>0, 1\leq u \leq m_a^{(i)}\}
%\]
%(which are parametrized by the rows of $\mu^{(i)}$'s) as follows.
For each $i \in I$, reindex (arbitrarily) the variables $\{x_r^{(i)} \mid 1\leq r \leq m_i\}$ into 
\[ \{x_{a,u,v}^{(i)}\mid a>0, 1\leq u \leq m_a^{(i)}, 1\leq v \leq a\},
\]
%(note that $\sum_a am_a^{(i)} = |\mu^{(i)}| = m_i$).
which are parametrized by the boxes of $\mu^{(i)}$.
Then let $\varphi_{\bm{\mu}}\colon \mU_{\ggg} \to \C(\bm{y}_{\bm{\mu}})$ be the linear homomorphism naturally defined from the map
\[ \bm{x}_\ggg \to \bm{y}_{\bm{\mu}} \colon x_{a,u,v}^{(i)} \mapsto y_{a,u}^{(i)},
\]
%\[ \varphi_{\bm{\mu}}(x_{a,u,v}^{(i)}) = y_{a,u}^{(i)},
%\]
%For each $i \in I$, choose arbitrarily a decomposition of the set of variables $\{x_r^{(i)} \mid 1\leq r \leq m_i\}$ into a disjoint union
%of $\ell(\mu^{(i)})$ sets;
%\[ \{x_r^{(i)}\mid 1\leq r\leq m_r\} = \bigsqcup_{\begin{smallmatrix}a>0\\ 1\leq u \leq m_r^{(i)}\end{smallmatrix}} \bm{x}_i(a,u),
%\]
%where the cardinality of $\bm{x}_i(a,u)$ is $a$.
%Then a linear homomorphism $define $\varphi_{\bm{\mu}}\colon \mU_\ggg$
%Let $\varphi_{\bm{\mu}}\colon \bm{x}_\ggg \to \bm{y}$ be the map sending all the variables in $\bm{x}_i(a,u)$ 
%to $y_{a,u}^{(i)}$, 
%and extend it to a linear homomorphism $\varphi_{\bm{\mu}}\colon\mU_\ggg \to \C(\bm{y})$ in the natural way.
%Note that this homomorphism does not depend on the reindexing due to the symmetry.
which does not depend on the reindexing due to the symmetry.

%Let $\C(\bm{y}_{\ell(\bm{\mu})})$ be the space of rational functions in the variables $\bm{y}_{\ell(\bm{\mu})} = 
%\big\{y^{(i)}_k \bigm| i \in I, 1\leq k \leq \ell(\mu^{(i)})\big\}$.
%We define a specialization map $\varphi_{\bm{\mu}}$ from 
%$\bm{x}_\ggg = \{x_r^{(i)}\mid i \in I, 1\leq r \leq m_i\}$ to $\bm{y}_{\ell(\bm{\mu})}$ by
%\[ \varphi_{\bm{\mu}}(x_r^{(i)}) = y_k^{(i)} \text{ \ for } r = \sum_{l=1}^{k-1} \mu_{l}^{(i)} + 1, \sum_{l=1}^{k-1} \mu_{l}^{(i)}
%   +2 ,\ldots,  \sum_{l=1}^k\mu_{l}^{(i)},
%\]
%and extend it to a linear map $\varphi_{\bm{\mu}}\colon \mathcal{U}_\ggg \to \C(\bm{y}_{\ell(\bm{\mu})})$ in the obvious way.

Define a lexicographic ordering $\leq$ on the set of $I$-tuples of partitions by $\bm{\mu} < \bm{\nu}$ if and only if
there exists $i \in I$ such that $\mu^{(j)} = \nu^{(j)}$ for $j < i$ and $\mu^{(i)} < \nu^{(i)}$.
Here the ordering on partitions are the usual lexicographic one.
Let
\[ \gG_{\bm{\mu}} = \bigcap_{\bm{\nu} > \bm{\mu}} \mathrm{ker}\, \varphi_{\bm{\nu}}\subseteq \mathcal{U}_\ggg, 
\]
which defines a filtration $\mU_\ggg = \bigcup_{\bm{\mu}} \gG_{\bm{\mu}}$.
We also define $\gG_{\bm{\mu}}' = \bigcap_{\bm{\nu} \geq \bm{\mu}} \mathrm{ker} \,\varphi_{\bm{\nu}} \subseteq \mU_{\ggg}$.
The zeros and poles of the functions in the image $\varphi_{\bm{\mu}} (\gG_{\bm{\mu}}) \cong \gG_{\bm{\mu}} /\gG_{\bm{\mu}}'$ 
are described by the 
following lemma.
For the proof, see \cite[Appendix A]{MR2290922}.

\begin{Lem}\label{Lem:description_of_zero-pole}
 Assume that $g(\bm{x}) \in \gG_{\bm{\mu}}$. \\
 {\normalfont(i)} 
  The function $\varphi_{\bm{\mu}}(g(\bm{x}))$ has a zero of order at least $2\min\{a,b\}$ whenever $y_{a,u}^{(i)} = y_{b,v}^{(i)}$.\\
 {\normalfont(ii)}% Let $i,j \in I$ $(i \neq j)$, $1\leq k \leq \ell(\mu^{(i)})$ and $1 \leq k' \leq \ell(\mu^{(j)})$.
  For $i, j \in I$ with $i \neq j$,
  $\varphi_{\bm{\mu}}(g(\bm{x}))$ has a pole of order at most $\mathrm{min}\{|c_{ji}|a, |c_{ij}|b\}$
  whenever $y_{a,u}^{(i)} = y_{b,v}^{(j)}$.
\end{Lem}

By this lemma, we see that $\varphi_{\bm{\mu}}(g(\bm{x}))$ for $g(\bm{x}) \in \gG_{\bm{\mu}}$ is of the form
\begin{equation}\label{eq:description_of_im}
 \frac{\prod_{i\in I}\prod_{(a,u)<(b,v)} (y_{a,u}^{(i)}-y_{b,v}^{(i)})^{2\min\{a,b\}}}
 {\prod_{i<j}\prod_{(a,u),(b,v)} (y_{a,u}^{(i)}-y_{b,v}^{(j)})^{\min\{|c_{ji}|a,|c_{ij}|b\}}}\cdot h_0(\bm{y}_{\bm{\mu}}),
\end{equation}
where $h_0(\bm{y}_{\bm{\mu}})$ is a Laurent polynomial in $\bm{y}_{\bm{\mu}}$ and 
symmetric under the exchange of variables $y_{a,u}^{(i)} \leftrightarrow y_{a,v}^{(i)}$.
%Here we fix for each $i \in I$ an arbitrary total ordering $\le$ on the set $\{(a,u)\mid a>0, 1\leq u \leq m_a^{(i)}\}$.

\subsection{Simplicity of poles}\label{subsection:simplicity}

The purpose of this subsection is to prove the following lemma.

\begin{Lem}\label{Lem:simplicity_of_poles}
 Let $\ggg \in Q^+$, and $i_1,i_2,\ldots$ be a sequence of elements of $I$.
 Define a sequence of variables $x_1,x_2,\ldots$ by $x_r = x^{(i_r)}_{\#\{s \leq r\mid\, i_s=i_r\}}$.
% Set $\gb = \ga_{i_1}+\cdots +\ga_{i_h}$, and rename the variables $\bm{x}_\gb$ into $\{x_1,\ldots,x_h\}$ as {\normalfont(\ref{eq:rename})}.
 Take $g_1(\bm{x}) \in \mU_\ggg$ arbitrarily, and define functions $g_r(\bm{x})$ $(r=1,2,\ldots)$ inductively by
 \[ g_r(\bm{x}) = \Res_{x_{r-1}=x_r}\, g_{r-1}(\bm{x}).
 \]
 Then each $g_r(\bm{x})$ has at most a simple pole at $x_r=x_{r+1}$.
\end{Lem}

%Throughout this subsection we fix $\ggg=\sum_i m_i\ga_i \in Q^+$, and write 
%for notational convenience. 

%The Laurent polynomial algebra $\C[x^{\pm 1} \mid x \in \bm{x}_\ggg]$ is naturally $\Z$-graded by the total degree.
%Since the specializations $x_1^{(i)} = \cdots = x_{\hat{c}_{ij}}^{(i)}=x_1^{(j)}$ and permutations preserve the degrees,
%$\mU_\ggg$ is also $\Z$-graded, where 
%\[ \deg \big(g_1/\gD_\ggg\big) =  \deg\, g_1 - \sum_{\begin{smallmatrix} i<j \\ c_ {ij}<0
%   \end{smallmatrix}} m_im_j.
%\]
%Let $\C[\bm{x}_\ggg]'$ denote the subalgebra of the polynomial algebra $\C[\bm{x}_\ggg]$ generated by the elements 
%$(x_r^{(i)}- x_{r'}^{(j)})$ with $(i,r) \neq (j,r')$,
%We regard $\C[\bm{x}_\ggg]'$ as a $\Z$-graded algebra by the natural degree grading:
%\[ \deg(x_r^{(i)} - x_s^{(j)})=1 \ \ \ \big((i,r) \neq (j,s)\big).
%\]
%and set
%\[ \mU_\ggg' = \mU_\ggg \cap \big\{ g(\bm{x})/\gD_\ggg\bigm| g(\bm{x})\in \C[\bm{x}_\ggg]'\big\}.
%\]
%Since $g_1|_{x_1^{(i)}=\cdots=x_{\hat{c}_{ij}}^{(i)}=x_1^{(j)}}=0$ is equivalent to 
%$g_1|_{x_1^{(i)}-x_1^{(j)} = \cdots = x_{\hat{c}_{ij}}-x_1^{(j)}=0}=0$ and permutations preserve the degrees,
%Obviously $\mU_\ggg'$ is a $\Z$-graded subspace with respect to the total degree.
%The following lemma is important for the proof of Proposition \ref{Prop:bilinear_well}.
Let $\C[\bm{x}_\ggg]$ denote the polynomial algebra in $\bm{x}_\ggg$, and define a subspace $\mU_\ggg' \subseteq \mU_\ggg$ by
\[ \mU_\ggg' = \mU_\ggg \cap \big\{ g'(\bm{x})/\gD_\ggg\bigm| g'(\bm{x})\in \C[\bm{x}_\ggg]\big\}.
\]
Obviously $\mU_\ggg'$ is a $\Z$-graded subspace with respect to the total degree.

\begin{Lem}\label{Lem:lower_bound_for_height}
 Assume that $\ggg \neq 0$ and $g(\bm{x}) \in \mU_\ggg'$ is homogeneous. Then we have
 \[ \deg g(\bm{x}) > -\het(\ggg).
 \]
\end{Lem}

\begin{proof}
 Take $\bm{\mu}=(\mu^{(1)},\ldots,\mu^{(n)})$ so that $g(\bm{x}) \in \gG_{\bm{\mu}}\setminus \gG_{\bm{\mu}}'$, 
 and let $m_a^{(i)}$ be the number of rows of length $a$ 
 in $\mu^{(i)}$.
 The image $\varphi_{\bm{\mu}}(\gG_{\bm{\mu}})$ is also $\Z$-graded by the total degree,
 and it is enough to show that
 \[ \deg \varphi_{\bm{\mu}}\big(g(\bm{x})\big) > -\het(\ggg)
 \]
 since $\varphi_{\bm{\mu}}$ preserves the degree.
 Since $g(\bm{x}) \in \mU'_\ggg$, it is clear that $\varphi_{\bm{\mu}}\big(g(\bm{x})\big)$ is of the form (\ref{eq:description_of_im})
 with $h_0(\bm{y}_{\bm{\mu}})$ being a polynomial.
 Hence we have
% \[ \deg \varphi_{\bm{\mu}}\big(g(\bm{x})\big) \geq 2\sum_{i \in I} \sum_{(a,u)<(b,v)} \min\{a,b\}- \sum_{i<j} \sum_{\begin{smallmatrix}
%   (a,u),(b,v)\end{smallmatrix}} \min\{|c_{ji}|a,|c_{ij}|b\},
% \]
 \begin{align*}
  \deg \varphi_{\bm{\mu}}\big(g(\bm{x})\big)\geq 2\sum_{i \in I}\sum_{(a,u)<(b,v)} \min\{a,b\} - \sum_{i<j}\sum_{(a,u), (b,v)}
  \min\{|c_{ji}|a,|c_{ij}|b\},   
 \end{align*}
% where $R(\mu^{(i)}) = \big\{(a,u)\bigm| a>0, 1\leq u \leq m_a^{(i)}\big\}$ is the parametrizing set of rows of $\mu^{(i)}$,
% \begin{align*}
%  \deg \varphi_{\bm{\mu}}\big(g(\bm{x})\big)\geq \sum_{i\in I}\Big(\sum_a am_a^{(i)}(m_a^{(i)}-1) &+ 2 \sum_{a<b} am_a^{(i)}m_b^{(i)}\Big)\\
%  &-\sum_{\begin{smallmatrix}i<j\\a,b \end{smallmatrix}}m_a^{(i)}m_b^{(j)} \min\{|c_{ji}|a,|c_{ij}|b\},
% \end{align*}
 and the right-hand side is larger than $-\het(\ggg)$ by Lemma \ref{Lem:degree_lower_bound} given below.
 Hence the assertion is proved.
\end{proof}

%Now we obtain the following useful formula.

%\begin{Prop}\label{Prop:convenient formular}
% Let $i_1,\ldots,i_h \in I$ and $k_1,\ldots,k_h \in \Z_{>0}$, and put $\gb=\ga_{i_1}+\cdots + \ga_{i_h}$.
% We rename the variables $\bm{x}_{\gb}$ into $\{x_1,\ldots,x_h\}$ by 
% \begin{equation}\label{eq:rename}
%  x_r = x^{(i_r)}_{\#\{1\leq s \leq r \mid i_s=i_r\}}.
% \end{equation}
% Then for any $g(\bm{x}) \in \mU$, it follows that
% \begin{equation}\label{eq:convenient}
%  \begin{split}
%   [R_{i_h,k_h},&[R_{i_{h-1},k_{h-1}},\ldots,[R_{i_2,k_2},R_{i_1,k_1}]\ldots]]\,g(\bm{x}) \\
%   &=\Res_{x_h} \Big\{\prod_{r=1}^{h-1} (x_{r+1}-x_r) g(\bm{x})\Big|_{ x_1=x_2=\cdots=x_h}\cdot (x_h)^{k_1+\cdots+k_h-h}\Big\}.
%  \end{split}
% \end{equation}
%if $f_\ga =[f_{i_h},[f_{i_{h-1}},\ldots,[f_{i_2},f_{i_1}]\!\ldots]]$,
%then for $g(\bm{x}) \in \mU$ we have
%\begin{equation}\label{eq:convenient}
% \begin{split}
%  \langle  f_{\ga,k}, g(\bm{x})\rangle = \Res_{x^{(i_m)}_{r_m}} \Big\{\prod_{a=2}^{m} 
%  (x_1^{(i_1)}-x_{r_{a}}^{(i_{a})}) 
%  g(\bm{x})\Big|_{ x^{(i_1)}_{r_1}=x^{(i_2)}_{r_2} =\cdots= x^{(i_m)}_{r_m}}\cdot (x^{(i_m)}_{r_m})^{k-m}\Big\},
%  \langle  f_{\ga,k}, g(\bm{x})\rangle = \Res_{y_h} \Big\{\prod_{a=2}^{h} 
%  (y_1-y_a) 
%  g(\bm{x})\Big|_{ y_1=y_2=\cdots=y_h}\cdot (y_h)^{k-h}\Big\},
% \end{split}
%\end{equation}
%where we write $(y_1,\ldots,y_h)$ for $(x_1^{(i_1)},\ldots,x_{m_{i_h}}^{(i_h)})$.
%In particular, if the sum of the orders of the poles of $g(\bm{x})$ at $y_a=y_b$ ($1\leq a < b \leq h$) is less than 
%$h-1$, then the left-hand side of (\ref{eq:convenient}) is $0$.
%\end{Prop}

\begin{proof}[Proof of Lemma \ref{Lem:simplicity_of_poles}]
 We show the lemma by the induction on $r$. The case $r=1$ is trivial.
 Assume that $r>1$, and set $\gb = \ga_{i_1}+\ga_{i_2}+\cdots + \ga_{i_{r+1}}$.
 We may assume that $\ggg - \gb \in Q^+$, since otherwise $g_r(\bm{x})$ does not contain the variable $x_{r+1}$ and hence the assertion
 trivially holds.
 Then, since 
 \[ \mU_\ggg \subseteq \C[x, (x-x')^{-1}\mid x \in \bm{x}_\ggg\setminus \bm{x}_\gb,\, x'\in \bm{x}_\ggg]\cdot
    \mU_\gb, 
 \]
 it is enough to show the assertion in the case $\ggg = \gb$.
 Therefore we assume $\ggg = \gb$ in the sequel.

 Without loss of generality we may assume that $g_1(\bm{x})$ is homogeneous, 
 and since multiplying a monomial $(x_1x_2\ldots x_{r+1})^a$ preserves $\mU_\gb$ and 
 does not affect the orders of the poles we are considering,
 we may further assume that $g_1(\bm{x}) \in \mU'_\gb$.
 Write $g_1(\bm{x}) =  g_1'(\bm{x})/ \gD_\gb$, and set $N = \deg g_1'(\bm{x})$.

 Let $\C[\bm{x}_\gb]'$ denote the $\C$-subalgebra of $\C[\bm{x}_\gb]$ generated by the vectors $(x_k-x_l)$ 
 for $1\leq k < l \leq r+1$.
 Since $\C[\bm{x}_\gb] = \C[x_{r+1}]\cdot \C[\bm{x}_\gb]'$ holds, $g_1'(\bm{x})$ is uniquely written as
 \begin{equation*}\label{eq:linear_comb}
  g_1'(\bm{x}) = \sum_{k =k_0}^N x_{r+1}^{N-k} q_k(\bm{x}),
 \end{equation*}
 where $q_k(\bm{x})$ is a homogeneous polynomial in $\C[\bm{x}_\gb]'$ with degree $k$, and $q_{k_0}(\bm{x}) \neq 0$.
 We claim that $q_{k_0}(\bm{x})/\gD_\gb \in \mU_\gb'$.
 Indeed, it is easy to check that $g_1'(\bm{x})\Big|_{x_1^{(i)}=\cdots=x_{\hat{c}_{ij}}^{(i)}=x_1^{(j)}} = 0$ implies
 $q_{k}(\bm{x})\Big|_{x_1^{(i)}=\cdots=x_{\hat{c}_{ij}}^{(i)}=x_1^{(j)}} = 0$ for all $k$.
 Moreover,
 the symmetry of $g_1'(\bm{x})$ under the exchange of variables $x_s^{(i)} \leftrightarrow x_{s'}^{(i)}$
 implies the same symmetry on $q_{k_0}(\bm{x})$, since
 \[ x_s^{N-k} =\big(x_{r+1}-(x_{r+1}-x_s)\big)^{N-k}\in x_{r+1}^{N-k} + \C[x_{r+1}]\cdot \big(\C[\bm{x}_\gb]'\big)^{>0}.
 \]
 Hence the claim is proved.

 Therefore we have $\deg \big(q_{k_0}(\bm{x})/\gD_\gb\big) \geq -r$ by Lemma \ref{Lem:lower_bound_for_height}, which implies
 \begin{equation} \label{eq:degree_inequality}
  \deg \big(q_k(\bm{x})/\gD_\gb\big) \geq -r \ \ \text{ for all } k \text{ such that } q_k(\bm{x}) \neq 0.
 \end{equation} 
 Note that, by the induction hypothesis, we have 
 \begin{equation*}\label{eq:def_of_gr}
  g_r(\bm{x}) = \prod_{s=2}^r(x_{s-1}-x_{s}) g_1(\bm{x})\Big|_{x_1=\cdots = x_{r}}= \sum_{k= k_0}^N x_{r+1}^{N-k}\Big(\prod_{s=2}^r(x_{s-1}-x_s) 
  q_k(\bm{x})/\gD_\gb\Big)\Big|_{x_1=\cdots = x_{r}}.
 \end{equation*}
 Since $q_k(\bm{x}) \in \C[\bm{x}_\gb]'$, we have
 \[ \Big(\prod_{s=2}^r(x_{s-1}-x_s) q_k(\bm{x})/\gD_\gb\Big)\Big|_{x_1=\cdots = x_{r}} \in \C[(x_r-x_{r+1})^{\pm 1}]
 \]
 for all $k$, and its degree is equal to or more than $-1$ by (\ref{eq:degree_inequality}). 
 Hence the assertion is proved.
% We may assume that $g_1(\bm{x})$ is homogenous.
% Write $g_1(\bm{x}) =  g_1'(\bm{x})/ \gD_\gb$,
% and set $N = \deg g_1'(\bm{x})$. 
% Since multiplying a monomial $(x_1x_2\ldots x_{r+1})^a$ preserves $\mU_\gb$ and 
% does not affect the orders of the poles we consider,
% we may further assume that $g_1'(\bm{x})$ is a polynomial.
% Since $\C[\bm{x}_\gb] = \C[x_1]\otimes \C[\bm{x}_\gb]'$, we can uniquely write
% \begin{equation}\label{eq:linear_comb}
%  g_1'(\bm{x}) = \sum_{k =0}^N x_1^{N-k} q_k(\bm{x}),
% \end{equation}
% where $q_k(\bm{x})$ is a homogeneous polynomial in $\C[\bm{x}_\gb]'$ with degree $k$.
% Let $k_0$ be the minimal degree such that $q_{k_0}(\bm{x}) \neq 0$.
% We easily see that $g'_1(\bm{x})/\gD_\gb \in \mU_\gb$ implies $q_{k_0}(\bm{x}) / \gD_\gb\in \mU_\gb'$,
% and hence it follows from Lemma \ref{Lem:lower_bound_for_height} that 
% \[ k_0 \geq 1 - \het(\gb) +\sum_{i<j}m_im_j= -r+\sum_{i<j}m_im_j.
% \]
% Therefore for any nonzero $q_k(\bm{x})$, the degree of $\prod_{s=2}^r (x_s-x_{s-1})q_k(\bm{x})/\gD_\gb$ is 
% equal to or more than $-1$.
% This, together with (\ref{eq:def_of_gr}) and (\ref{eq:linear_comb}), obviously implies the assertion, and the proof is complete
% Since 
% \[ \Big\{\prod_{2\leq s \leq r}(x_{s-1}-x_s)q_k(\bm{x})/\gD\Big\}\Big|_{x_1=\cdots=x_{r+1}} \in \C\big[(x_{r+1}-x_r)^{\pm 1}\big]
% \] 
% and the specialization $x_1=\cdots =x_r$ preserves the degree,
% we easily see that this implies the assertion. The proof is complete.
\end{proof}

It remains to prove the following.

\begin{Lem}\label{Lem:degree_lower_bound}
 Let $\bm{\mu} = (\mu^{(1)},\ldots,\mu^{(n)})$ be an $I$-tuple of partitions such that $\mu^{(i)}$ has $m_a^{(i)}$ rows of length $a$,
 and assume that at least one of the partitions is nonzero. 
 Set $R(\mu^{(i)}) = \big\{(a,u) \bigm| a>0, 1\leq u \leq m_a^{(i)}\big\}$, and define
 \begin{align*}
  P(\bm{\mu})= 2\sum_{i \in I}\sum_{\begin{smallmatrix}(a,u),(b,v) \in R(\mu^{(i)}); \\ (a,u)<(b,v) 
  \end{smallmatrix}} \min\{a,b\} - \sum_{i<j}\sum_{\begin{smallmatrix} (a,u) \in R(\mu^{(i)}) \\ (b,v) \in R(\mu^{(j)})\end{smallmatrix}}
  \min\{|c_{ji}|a,|c_{ij}|b\}.   
 \end{align*}
% \begin{equation*}\label{eq:Combinatorial_inequality}
%  \begin{split}
%   P(\bm{\mu}) = \sum_{i\in I}\Big(\sum_a am_a^{(i)}(m_a^{(i)}-1) + 2 \sum_{a<b} am_a^{(i)}m_b^{(i)}\Big)
%   -\sum_{\begin{smallmatrix}i<j \\ a,b \end{smallmatrix}}%\sum_{a,b}m_a^{(i)}m_b^{(j)}  
%   m_a^{(i)}m_b^{(j)}\min\{|c_{ji}|a,|c_{ij}|b\},
%  2\sum_{i\in I}\sum_{(a,u)<(b,v)} \min\{a,b\} - \sum_{i<j}\sum_{\begin{smallmatrix}(a,u),(b,v)\end{smallmatrix}} \min\{|c_{ji}|a,|c_{ij}|b\} 
%   &\geq -\sum_i |\mu^{(i)}| +1.
%  \end{split}
% \end{equation*} 
 Then we have $P(\bm{\mu}) > -\sum_i |\mu^{(i)}|$.
\end{Lem}

\begin{proof}
% Denote the left-hand side of (\ref{eq:Combinatorial_inequality}) by $P(\bm{\mu})$.
 We prove the assertion by the induction on $\sum_i |\mu^{(i)}|$. The case $\sum_i |\mu^{(i)}| = 1$ is easily checked.

 Assume that $\sum_i |\mu^{(i)}| > 1$.
% Let $m_a^{(i)}$ denote the length of $a$-th column of $\mu^{(i)}$, that is, $m_a^{(i)} = \#\{k \mid \mu_k^{(i)} \geq a\}$.
% number of rows of $\mu^{(i)}$ with length $a$.
 For $i \in I$, set
 \[ c_i = \begin{cases} 1 & \text{if $\ga_i$ is long}, \\
                        2 & \text{if $\fg$ is of type $BCF$ and $\ga_i$ is short},\\
                        3 & \text{if $\fg$ is of type $G_2$ and $\ga_i$ is short,}\end{cases}
 \]
 where we say every $\ga \in R^+$ is long when $\fg$ is of type $ADE$.
% $c_i = 1$ if $\ga_i$ is long, $c_i =2$ if $\fg$ is of type $BCF$ and $\ga_i$ is short, 
% and $c_i = 3$ if $\fg$ is of type $G_2$ and $\ga_i$ is short.
 Let 
 \[ L = \max\big\{ \lceil \mu_1^{(i)}/c_i \rceil\bigm| i\in I\big\},
 \]
 where $\lceil a \rceil$ is the smallest integer equal to or larger than
 $a$. Let $N_i$ ($i \in I$) denote the number of boxes in $\mu^{(i)}$ strictly right to the $c_i(L-1)$-th column;
 \[ N_i = m_L^{(i)} \ (c_i =1), \ \  N_i=2m_{2L}^{(i)} + m_{2L-1}^{(i)}\ (c_i=2), \ \  N_i= 3m_{3L}^{(i)} + 2m_{3L-1}^{(i)}+m_{3L-2}^{(i)}\ 
    (c_i = 3).
 \]
% which is the number of boxes strictly right to the $c_i(l-1)$-th column.
 Set $\gb = \sum_i N_i \ga_i\in Q^+$. % by
% \[ \ggg = \sum_{\begin{smallmatrix}i \in I\\ c_i =1\end{smallmatrix}} m_l^{(i)}\ga_i + 
%    \sum_{\begin{smallmatrix}i \in I\\ c_i =2\end{smallmatrix}} 
%    (m_{2l-1}^{(i)} + 2m_{2l}^{(i)})\ga_i + \sum_{\begin{smallmatrix}i \in I\\ c_i =3\end{smallmatrix}} 
%    (3m_{3l}^{(i)} + 2m_{3l-1}^{(i)}m_{3l-2}^{(i)}+)\ga_i.
% \]
 Since $\fg$ is of finite type, there exists $j \in I$ such that $\langle \ga_j^\vee, \gb\rangle > 0$.
 Fix such an index $j$. We claim that there exists a partition $\mu'$ of $|\mu^{(j)}|-1$ such that 
 \[ P(\bm{\mu}) - P(\bm{\mu}') \geq -1,
 \]
 where $\bm{\mu}'$ is the $I$-tuple of partitions obtained from $\bm{\mu}$ by replacing $\mu^{(j)}$ with $\mu'$.
 By the induction hypothesis, this completes the proof.

 Let us prove the claim. First assume that $c_j =1$. 
 Note that $m_L^{(j)}>0$ holds since $\langle \ga_j^\vee, \gb\rangle > 0$.
 Let $\mu'$ be the partition obtained from $\mu^{(j)}$ by removing one box in the $L$-th column.
 Then it is directly checked that
 \begin{align*}
  P(\bm{\mu}) - P(\bm{\mu}') = 2(m_L^{(j)} -1) - \sum_{\begin{smallmatrix}i \in I;\\ c_ {ij}<0\end{smallmatrix}} N_i
  = \langle \ga_j^\vee, \gb\rangle -2\geq -1,
 \end{align*}
 where $\bm{\mu}'$ is defined as above. 
 Hence the claim is proved in this case.

 Next assume that $c_j = 2$,
 and let us further assume that both $m_{2L-1}^{(j)}$ and $m_{2L}^{(j)}$ are nonzero.
 For $k = 0,1$, let $\mu'[k]$ denote the partition obtained from $\mu^{(j)}$
 by removing one box in the $(2L-k)$-th column,
 and set $\bm{\mu}'[k]$ to be the $I$-tuple of partitions obtained from $\bm{\mu}$ by replacing $\mu^{(j)}$ with $\mu'[k]$.
 Then it follows that
 \[ P(\bm{\mu}) - P(\bm{\mu}'[k]) = 2(m_{2L}^{(j)}+\gd_{k1} m_{2L-1}^{(j)}-1) - 
    \sum_{\begin{smallmatrix}i \in I;\\ c_ {ji}=-2\end{smallmatrix}} m_L^{(i)}
    - \sum_{\begin{smallmatrix}i \in I;\\ c_ {ji}=-1\end{smallmatrix}} (m_{2L}^{(i)}+\gd_{k1}m_{2L-1}^{(i)}).
 \]
 Hence we have
 \[ \big(P(\bm{\mu})-P(\bm{\mu}'[0])\big) + \big(P(\bm{\mu})-P(\bm{\mu}'[1])\big)= \langle \ga_j^\vee, \gb\rangle-4 >-4,
 \]  
 which implies that there is $k \in \{0,1\}$ such that $P(\bm{\mu}) - P(\bm{\mu}'[k])\geq -1$.
 When $m_{2L-k}^{(j)}=0$ for either $k=0$ or $k=1$, we can show similarly that $P(\bm{\mu})-P(\bm{\mu}'[k']) \geq -1$ for $k' \neq k$.
% When $\mu'[k]$ is not defined for some $k \in \{0,1\}$, which implies $\mu'[k']$ is well-defined where $k'\neq k$, 
% we can show similarly that $P(\bm{\mu})-P(\bm{\mu}'[k']) \geq -1$. 
 Hence the claim is verified in this case.

 The case $c_j = 3$ is proved similarly, and we omit the detail. 
\end{proof}

\subsection{Proof of Proposition \ref{Prop:dual_isom}}\label{subsection:pf_of_dual}

We will define functions $g_{\ggg,k}(\bm{x}) \in \ol{\mU}_\ggg$ for $\ggg \in R^+$ and $k \in \Z_{>0}$.
Let $(\ , \ )$ denote the unique nondegenerate $W$-invariant symmetric bilinear form on $P$ normalized so that the square length of $\gt$ is
$2$.

First assume that $\fg$ is not of type $G_2$, and 
set $\Ish = \{ i \in I \mid \ga_i \text{ is short\,}\}$. % ($\Ish = \emptyset$ if $\fg$ is of type $ADE$).
%For $m \in \Z_{\ge 0}$, denote by $[m]$ the set $\{1,2,\ldots,m\}$,
%and let $[m]_{\mathrm{e}}$ (resp.\ $[m]_{\mathrm{o}}$) denote the subset of $[m]$ consisting of even (resp.\ odd) numbers.
Fix $\ggg = \sum_i m_i\ga_i \in R^+$, and write $\ol{\ggg} = \sum_{ i\in \Ish;\, m_i\notin 2\Z}\ga_i$.
It follows that $(\ggg, \ggg) + (\ol{\ggg},\ol{\ggg}) = 2$.
%Let $x \in \bm{x}_\ggg$ be an arbitrary variable. 
Let
\[ h_\ggg(\bm{x}) = \!\!\prod_{i \in I \setminus \Ish} \prod_{r<s} (x_r^{(i)}-x_s^{(i)})^2 \prod_{i \in \Ish}\!\prod_{\begin{smallmatrix}
   r<s \\ r,s: \text{\,even,\,or}\\ r,s: \text{\,odd \ \ \ \ }\end{smallmatrix}}\!\!\!
   (x_r^{(i)}-x_s^{(i)})^2 \prod_{\begin{smallmatrix}i,j \in \Ish; \\ i<j,\,c_ {ij}=-1 \end{smallmatrix}} \prod_{\begin{smallmatrix}
   r:\text{\,even},\,s:\text{\,odd,\,or}\\ r:\text{\,odd},\,s:\text{\,even\ \ }
%: \text{\,even,\,or}\\ r,s: \text{\,odd \ \ \ \ }
\end{smallmatrix}}\!\!\!(x_r^{(i)} -x_s^{(j)}).
\]
For $i \in I$ and $1\leq r \leq m_i$, put
\[ \ggg^{(i,r)} = \begin{cases} \ggg -(-1)^r \ol{\ggg} & \text{if $i \in I_{\mathrm{sh}}$},\\
%                             \ggg - \ol{\ggg} & \text{if $i \in I_{\mathrm{sh}}$ and $r$ is even},\\
                             \ggg             & \text{otherwise},
               \end{cases}
\]
and define
\[ g_{\ggg,k}(\bm{x}) = \mathrm{Sym}\Bigg( \frac{x^{-k+1}\prod_{i \in I}\prod_{r=1}^{m_i}(x_r^{(i)})^{-(\ga_i,\ggg^{(i,r)})}\cdot
   h_\ggg(\bm{x})}{\prod_{\begin{smallmatrix}i<j \\ c_{ij}<0\end{smallmatrix}}
   \prod_{r,s} \big(x_r^{(i)}-x_s^{(j)}\big)}\Bigg)
\] 
for $k \in \Z$, where $\mathrm{Sym}$ denotes the symmetrization over 
$n$ sets of variables $\big\{x_r^{(i)} \bigm| 1\leq r \leq m_i\big\}$ ($i \in I$), and $x \in \bm{x}_\ggg$ is an arbitrarily fixed variable. 
We easily check that $g_{\ggg,k}(\bm{x}) \in \mU_\ggg$.
Moreover, it follows from a direct calculation that
\begin{align*}\label{eq:degree_cal}
 \deg h_\ggg(\bm{x}) &= \sum_{i \in I\setminus \Ish} m_i(m_i-1) +\sum_{i \in \Ish} \left\lceil \frac{m_i(m_i-2)}{2}\right\rceil + 
 \sum_{\begin{smallmatrix}i,j \in \Ish;\\ i<j,\,c_ {ij}=-1\end{smallmatrix}} \bigg\lfloor \frac{m_im_j}{2} \bigg\rfloor\nonumber\\ 
 &=\frac{1}{2}\big\{(\ggg,\ggg) + (\ol{\ggg},\ol{\ggg})\big\} - \het (\ggg) + \sum_{\begin{smallmatrix}i<j \\ c_{ij}<0\end{smallmatrix}}
  m_im_j=1-\het (\ggg) +\sum_{\begin{smallmatrix}i<j \\ c_{ij}<0\end{smallmatrix}} m_im_j,
\end{align*}
and then we have by (\ref{eq:convenient}) that
\begin{equation}\label{eq:compatibility}
 \langle f_{\ggg,k}, g_{\ggg,k}(\bm{x})\rangle \in \C^\times, \ \ \ \langle f_{\ggg',k'}, g_{\ggg,k}(\bm{x})\rangle = 0 
 \text{ if }\ggg-\ggg' \notin Q^+
 \text{ or } \ggg=\ggg', k\neq k'.
\end{equation}
It is also proved from Lemma \ref{Lem:description_of_barU} that $g_{\ggg,k}(\bm{x}) \in \ol{\mU}_\ggg$ if $k > 0$.

%Moreover we have by a direct calculation that
%\begin{align*}
% \deg h_\ggg(\bm{x}) &= \sum_{i \in I\setminus \Ish} m_i(m_i-1) +\sum_{i \in \Ish} \left\lceil \frac{m_i(m_i-2)}{2}\right\rceil + 
% \sum_{\begin{smallmatrix}i,j \in \Ish;\\ i<j,\,c_ {ij}<0\end{smallmatrix}} \bigg\lfloor \frac{m_im_j}{2} \bigg\rfloor\\
% &=\frac{1}{2}\big\{(\ggg,\ggg) + (\ol{\ggg},\ol{\ggg})\big\} - \het (\ggg) + \sum_{\begin{smallmatrix}i<j \\ c_{ij}<0\end{smallmatrix}}
%  m_im_j=1-\het (\ggg) +\sum_{\begin{smallmatrix}i<j \\ c_{ij}<0\end{smallmatrix}} m_im_j,
%\end{align*}
%where we set $\ol{\ggg} = \sum_{ i\in \Ish;\, m_i\notin 2\Z}\ga_i$.
%Hence it follows from (\ref{eq:convenient}) that
%\begin{equation}\label{eq:compatibility}
% \langle f_{\ggg,k}, g_{\ggg,k}(\bm{x})\rangle \in \C^\times, \ \ \ \langle f_{\ggg',k'}, g_{\ggg,k}(\bm{x})\rangle = 0 
% \text{ if }\ggg-\ggg' \notin Q^+
% \text{ or } \ggg=\ggg', k\neq k'.
%\end{equation}
When $\fg$ is of type $G_2$, define $g_{\ggg,k}(\bm{x}) \in \ol{\mU}_\ggg$ for $\ggg \in R^+$ and $k \in \Z_{>0}$ by
\[ g_{\ggg,k}(\bm{x}) = \begin{cases} (x_1^{(i)})^{-k-1} & \text{if } \ggg= \ga_i,\\
                                      (x_1^{(1)})^{-k+m-1}\prod_{r=1}^m(x_r^{(2)})^{-1}/\gD_\ggg 
    & \text{if } \ggg = \ga_1+m \ga_2,\\
                                      \mathrm{Sym}\Big((x_1^{(1)})^{-k}(x_2^{(1)})^{-1}(x_1^{(1)}-x_2^{(1)})^2/ \gD_\ggg\Big) 
    & \text{if } \ggg = 2\ga_1+3\ga_2, 
%                              \mathrm{Sym}\Big(x^{-k+\het(\ggg)-1}/\gD_\ggg\Big) & \text{otherwise},
   \end{cases}
\]
where $\ga_1$ is long and $\ga_2$ is short.
%Let $g_{\ggg,k}(\bm{x}) = \mathrm{Sym}\left(x^{-k+\het(\ggg)-1}h_\ggg(\bm{x})\right)$
%where $x \in \bm{x}_\ggg$ is arbitrarily fixed.
We easily check that these $g_{\ggg,k}(\bm{x})$ also satisfy (\ref{eq:compatibility}).

Now we show Proposition \ref{Prop:dual_isom}.
The proof is carried out in a similar line as \cite[Proposition 3.1.3]{MR1934307}, in which the case $\fg = \mathfrak{sl}_3$ is treated.

\begin{proof}[Proof of Proposition \ref{Prop:dual_isom}]
 The injectivity is proved as follows. 
 Let $0 \neq g(\bm{x}) \in \ol{\mU}_\ggg$, and consider it as a formal Laurent series by expanding
 all $(x^{(i)}_r -x^{(j)}_s)^{-1}$ ($i < j$) in positive power of 
 $x_s^{(j)}/x_r^{(i)}$. 
 If the coefficient of a monomial $\prod_r (x^{(1)}_r)^{k_r^{(1)}}\cdots \prod_r(x^{(n)}_r)^{k_r^{(n)}}$ in the series 
 is $c \neq 0$ , then taking $F = \prod_r f_{1,-k_r^{(1)}-1} \cdots \prod_r f_{n,-k_r^{(n)}-1}$, we have $\langle F, g(\bm{x})\rangle =c 
 \neq 0$.
 Hence the injectivity is proved.
 
 Next we prove that for any $F \in (\ol{U}^-_{\bL})_{-\ggg}$ there exists $g(\bm{x}) \in \ol{\mU}_\ggg$ such that 
 $\langle F, g(\bm{x})\rangle \neq 0$, which implies the surjectivity and completes the proof.
 We may assume $F \in U(t\fn_-[t])$ by the Poincar\'{e}-Birkhoff-Witt theorem.
 Fix a total ordering $\preceq$ on $R^+$ such that $\ga \preceq \gb$ holds if $\gb -\ga \in Q^+$.
 We also denote by $\preceq$ the lexicographic ordering on $R^+ \times \Z_{>0}$.
 Let $B$ be the Poincar\'{e}-Birkhoff-Witt basis of $U(t\fn_-[t])$ with respect to this ordering;
 \[ B=\big\{f_{\gb_N,k_N} \cdots f_{\gb_2,k_2}f_{\gb_1,k_1} \bigm| \gb_a \in R^+, k_a \in \Z_{>0}, (\gb_N,k_N) \preceq \cdots \preceq (\gb_1,k_1)
    \big\}.
 \]
 Write $F = \sum_{b \in B} c_b b$ with $c_b \in \C$, and let $b_0=f_{\gb_N,k_N} \cdots f_{\gb_2,k_2}f_{\gb_1,k_1}\in B$ be the minimum vector
 with respect to the right-to-left lexicographic order such that $c_{b_0} \neq 0$.
 Then define $g(\bm{x})\in \ol{\mU}_\ggg$ by
 \[ g(\bm{x}) = \mathrm{Sym}\Big(\psi_{\gb_{<N}}\big(g_{\gb_N,k_N}(\bm{x}_{\gb_N})\big)\cdots \psi_{\gb_{<2}}\big(g_{\gb_2,k_2}(\bm{x}_{\gb_2})
    \big)\cdot    g_{\gb_1,k_1}(\bm{x}_{\gb_1})\Big),
 \]
 where we set $\gb_{<a} = \gb_1+\cdots + \gb_{a-1}$, and  $\psi_{\gb}$ for $\gb = \sum_i m_i\ga_i$ to be the algebra homomorphism defined by
 $\psi_{\gb}(x_r^{(i)}) = x_{r+m_i}^{(i)}$.
 It is proved from (\ref{eq:compatibility}) that 
 \begin{equation*}\label{eq:property_of_monomial}
  \langle b_0, g(\bm{x}) \rangle \in \C^\times, \ \ \ \langle b', g(\bm{x})\rangle = 0 \ \ \text{for} \ b'\in B \ \text{such that}
    \ b' \succ b_0,
 \end{equation*}
 and hence we have $\langle F, g(\bm{x})\rangle \neq 0$, as required. The proof is complete.
\end{proof}

\subsection{Proof of Lemma \ref{Lem:fermionic_form}}\label{subsection:fermionic_form}

Let $\gl =\sum_k \ell_k \varpi_{i_k} \in P^+$ and $\ggg=\sum_i m_i\ga_i \in Q^+$ be as in Lemma \ref{Lem:fermionic_form}.
Let $\bm{\mu}$ be an $I$-tuple of partitions such that $|\mu^{(i)}| = m_i$ and $\mu^{(i)}$ has $m_a^{(i)}$ rows of length $a$,
and assume that $g(\bm{x}) \in \mV_\ggg \cap \gG_{\bm{\mu}}$.
By Lemma \ref{Lem:description_of_zero-pole} and the definition of $\mV_\ggg$,
the function $\varphi_{\bm{\mu}}\big(g(\bm{x})\big)$ is written as $\varphi_{\bm{\mu}}\big(g(\bm{x})\big) = h_0(\bm{y}_{\bm{\mu}})
h_1(\bm{y}_{\bm{\mu}})$, where
\[ h_0(\bm{y}_{\bm{\mu}})=\frac{
   \prod_{i\in I}\prod_{(a,u)<(b,v)} (y_{a,u}^{(i)}-y_{b,v}^{(i)})^{2\min\{a,b\}}}{\prod_{i \in I}\prod_{(a,u)}
   (y_{a,u}^{(i)})^{\sum_{k\in S_i}\min\{a,\ell_k\}}\prod_{i<j}\prod_{(a,u),(b,v)}
   (y_{a,u}^{(i)}-y_{b,v}^{(j)})^{\min\{|c_{ji}|a,|c_{ij}|b\}}},
\]
and $h_1(\bm{y}_{\bm{\mu}})$ is a polynomial
%in variables $\bm{y} = \{y_{a,u}^{(i)}\mid i\in I, a>0, 1\leq u \leq m_a^{(i)}\}$,
symmetric under the exchange $y_{a,u}^{(i)} \leftrightarrow y_{a,v}^{(i)}$.
Moreover by Lemma \ref{Lem:description_of_barU}, we have
\[ \deg_{y_{a,u}^{(i)}} \varphi_{\bm{\mu}}\big(g(\bm{x})\big) \leq -2a
\]
for all $i,a,u$.
Since $\deg_{y_{a,u}^{(i)}} h_0(\bm{y}_{\bm{\mu}}) = -p_a^{(i)} -2a$, we have $\deg_{y_{a,u}^{(i)}} h_1(\bm{y}_{\bm{\mu}}) 
\leq p_a^{(i)}$ for all $i,a,u$.
Hence it follows that $\varphi_{\bm{\mu}} \big (\mV_\ggg \cap \gG_{\bm{\mu}}\big) = 0$ if $p_a^{(i)} < 0$ for some $i,a$, and otherwise
\[ \dim \varphi_{\bm{\mu}} \big(\mV_\ggg \cap \gG_{\bm{\mu}}\big) \leq \prod_{\begin{smallmatrix}i \in I\\ a>0\end{smallmatrix}}
   \begin{pmatrix} p_a^{(i)} + m_a^{(i)}\\m_a^{(i)}\end{pmatrix}.
\]
Since $\varphi_{\bm{\mu}} (\mV_\ggg \cap \gG_{\bm{\mu}}) \cong (\mV_\ggg \cap\gG_{\bm{\mu}})/ (\mV_\ggg \cap \gG_{\bm{\mu}}')$ and
\[ \dim \mV_\ggg = \sum_{\bm{\mu}}\dim (\mV_\ggg \cap\gG_{\bm{\mu}})/ (\mV_\ggg \cap \gG_{\bm{\mu}}'),
\]
the lemma follows.

\section*{Acknowledgments}
The author is supported by JSPS Grant-in-Aid for Young Scientists (B) No.\ 25800006.

\newcommand{\etalchar}[1]{$^{#1}$}
\def\cprime{$'$} \def\cprime{$'$} \def\cprime{$'$} \def\cprime{$'$}

%\bibliographystyle{alpha}
%\bibliography{bibliography}

\end{document}